\documentclass[12pt]{article}
\usepackage{graphicx}
\usepackage{amsmath,amsthm,amssymb,enumerate}
\usepackage{euscript,mathrsfs}
\usepackage[left=2cm,right=2cm,top=3.5cm,bottom=3.5cm]{geometry}
\usepackage{color}

\usepackage{soul}

\catcode`\@=11 \@addtoreset{equation}{section}

\catcode`\@=12

\allowdisplaybreaks

\newtheorem{Theorem}{Theorem}[section]
\newtheorem{Proposition}[Theorem]{Proposition}
\newtheorem{Lemma}[Theorem]{Lemma}
\newtheorem{Corollary}[Theorem]{Corollary}

\theoremstyle{definition}
\newtheorem{Definition}[Theorem]{Definition}

\newtheorem{Remark}[Theorem]{Remark}

\newcommand{\bTheorem}[1]{
\begin{Theorem} \label{T#1} }
\newcommand{\eT}{\end{Theorem}}

\newcommand{\bProposition}[1]{
\begin{Proposition} \label{P#1}}
\newcommand{\eP}{\end{Proposition}}

\newcommand{\bLemma}[1]{
\begin{Lemma} \label{L#1} }
\newcommand{\eL}{\end{Lemma}}

\newcommand{\bCorollary}[1]{
\begin{Corollary} \label{C#1} }
\newcommand{\eC}{\end{Corollary}}

\newcommand{\bRemark}[1]{
\begin{Remark} \label{R#1} }
\newcommand{\eR}{\end{Remark}}
\newcommand{\E}{\mathcal{E}}

\newcommand{\bDefinition}[1]{
\begin{Definition} \label{D#1} }
\newcommand{\eD}{\end{Definition}}

\newcommand{\Del}{\Delta_x}

\newcommand{\Ds}{\mathbb{D}_x}

\newcommand{\Se}{\mathbb{S}_{\rm eff}}

\newcommand{\bfphi}{\boldsymbol{\varphi}}

\newcommand{\bFormula}[1]{
\begin{equation} \label{#1}}
\newcommand{\eF}{\end{equation}}

\newcommand{\Ov}[1]{\overline{#1}}

\newcommand{\DC}{C^\infty_c}

\newcommand{\aleq}{\stackrel{<}{\sim}}

\newcommand{\vr}{\varrho}
\newcommand{\vre}{\vr_\ep}

\newcommand{\vue}{\vu_\ep}
\newcommand{\tvr}{\tilde \vr}
\newcommand{\tvu}{{\tilde \vu}}

\newcommand{\vu}{\vc{u}}
\newcommand{\vm}{\vc{m}}

\newcommand{\vc}[1]{{\bf #1}}

\newcommand{\Div}{{\rm div}_x}
\newcommand{\Grad}{\nabla_x}

\newcommand{\dx}{\,{\rm d} {x}}

\newcommand{\dt}{\,{\rm d} t }

\newcommand{\dxdt}{\dx  \dt}
\newcommand{\intO}[1]{\int_{\Omega} #1 \ \dx}

\newcommand{\vv}{\vc{v}}

\newcommand{\D}{{\rm d}}

\newcommand{\ep}{\varepsilon}

\def\softd{{\leavevmode\setbox1=\hbox{d}%
          \hbox to 1.05\wd1{d\kern-0.4ex{\char039}\hss}}}
\definecolor{Cgrey}{rgb}{0.85,0.85,0.85}
\definecolor{Cblue}{rgb}{0.50,0.85,0.85}
\definecolor{Cred}{rgb}{1,0,0}
\definecolor{fancy}{rgb}{0.10,0.85,0.10}

\newcommand\Cbox[2]{%
    \newbox\contentbox%
    \newbox\bkgdbox%
    \setbox\contentbox\hbox to \hsize{%
        \vtop{
            \kern\columnsep
            \hbox to \hsize{%
                \kern\columnsep%
                \advance\hsize by -2\columnsep%
                \setlength{\textwidth}{\hsize}%
                \vbox{
                    \parskip=\baselineskip
                    \parindent=0bp
                    #2
                }%
                \kern\columnsep%
            }%
            \kern\columnsep%
        }%
    }%
    \setbox\bkgdbox\vbox{
        \color{#1}
        \hrule width  \wd\contentbox %
               height \ht\contentbox %
               depth  \dp\contentbox
        \color{black}
    }%
    \wd\bkgdbox=0bp%
    \vbox{\hbox to \hsize{\box\bkgdbox\box\contentbox}}%
    \vskip\baselineskip%
}

\date{}

\begin{document}


\title{Generalized solutions to models of compressible viscous fluids}

\author{Anna Abbatiello\thanks{The research of A.A. is supported by Einstein Foundation, Berlin.} \and   Eduard Feireisl\thanks{The research of E.F. leading to these results has received funding from the
Czech Sciences Foundation (GA\v CR), Grant Agreement
18-12719S.} \and Anton\' \i n Novotn\' y { \thanks{The work of A.N. was supported by Brain Pool program funded by the Ministry of Science and ICT through the National Research Foundation of Korea(NRF-2019H1D3A2A01101128).}}
}


\maketitle

\centerline{Institute of Mathematics, Technische Universit\"{a}t Berlin,}
\centerline{Stra{\ss}e des 17. Juni 136, 10623 Berlin, Germany}
\centerline{anna.abbatiello@tu-berlin.de}
\centerline{and}

\centerline{Faculty of Mathematics and Physics, Charles University}
\centerline{Sokolovsk\'a 83, CZ-186 75 Prague 8, Czech Republic}
\centerline{feireisl@math.cas.cz}
\centerline{and}

\centerline{Universit\'e de Toulon, IMATH, EA 2134, }
\centerline{BP 20132, 83957 La Garde, France}
\centerline{novotny@univ-tln.fr}

\begin{abstract}

We propose a new approach to models of general compressible viscous fluids based on the concept of dissipative solutions. These are weak solutions satisfying the underlying equations modulo a defect measure. A dissipative solution coincides with the strong solution 
as long as the latter exists (weak--strong uniqueness) and they solve the problem in the classical sense as soon as they are smooth
(compatibility). 
We consider general models of compressible viscous fluids with non--linear viscosity tensor and non--homogeneous boundary conditions, for which the problem of existence of global--in--time weak/strong solutions is largely open. 

\end{abstract}


\bigskip

\tableofcontents

\section{Introduction}
\label{P}

Fluids in continuum mechanics are characterized by Stokes' law 
\[
\mathbb{T} = \mathbb{S} - p \mathbb{I}
\]
relating the Cauchy stress $\mathbb{T}$ to the viscous stress $\mathbb{S}$ and the scalar pressure $p$. 
For the sake of simplicity, we ignore the thermal effects and consider the basic system of field 
equations for the fluid density $\vr = \vr(t,x)$ and the macroscopic velocity $\vu = \vu(t,x)$:

\begin{itemize}
\item {\bf conservation of mass}
\begin{equation} \label{P1}
\partial_t \vr + \Div (\vr \vu) = 0; 
\end{equation}
\item {\bf balance of linear momentum}
\begin{equation} \label{P2}
\partial_t (\vr \vu) + \Div (\vr \vu \otimes \vu) + \Grad p(\vr) = \Div \mathbb{S}; 
\end{equation}
\item {\bf balance of energy}
\begin{equation} \label{P3}
\partial_t \left( \frac{1}{2} \vr |\vu|^2 + P(\vr) \right) + 
\Div \left[ \left( \frac{1}{2} \vr |\vu|^2 + P(\vr) + p(\vr) \right) \vu \right] = \Div \left( \mathbb{S} \cdot \vu \right)
- \mathbb{S} : \Grad \vu,
\end{equation}
where $P(\vr)$ is the so--called pressure potential related to the pressure $p$, 
\begin{equation} \label{P4}
P'(\vr) \vr - P(\vr) = p(\vr).
\end{equation}
\end{itemize}

We suppose the fluid is contained in a bounded domain $\Omega \subset R^d$, $d = 2,3$, with general {\bf inflow--outflow boundary conditions}

\begin{equation} \label{P5}
\vu|_{\partial \Omega} = \vu_B,\ \vr|_{\Gamma_{in}} = \vr_B,\ 
\Gamma_{in} = \left\{ x \in \partial \Omega \ \Big| \ \vu_B \cdot \vc{n} < 0 \right\},
\end{equation}
where $\vc{n}$ is the outer normal vector to $\partial \Omega$. The viscous stress tensor $\mathbb{S}$ is related to the 
symmetric velocity gradient 
\[
\Ds \vu = \frac{\Grad \vu + \Grad \vu^t}{2}
\]
through a general {\bf implicit rheological law} 
\begin{equation} \label{P6}
\Ds \vu : \mathbb{S} = 
F(\Ds \vu) + F^*(\mathbb{S}) 
\end{equation}
for a suitable convex potential $F$ and its conjugate $F^*$. Finally, we prescribe the initial conditions
\begin{equation} \label{P7}
\vr(0, \cdot) = \vr_0, \ \vr \vu(0, \cdot) = \vm_0. 
\end{equation}

The total energy of the fluid is not conserved due to the presence of the dissipative term $\mathbb{S}: \Grad \vu$ on the right--hand side of \eqref{P3}. As a matter of fact, the equation \eqref{P3} can be deduced from \eqref{P1}, \eqref{P2} via a straightforward manipulation as long as all quantities involved are smooth enough. In the weak formulation used in the present paper, the 
equation \eqref{P3} is replaced by a suitable form of energy inequality discussed below. 

If $F$ is a proper convex l.s.c. function, the rheological relation \eqref{P6} can be interpreted 
in view of Fenchel--Young inequality as 
\[
\mathbb{S} \in \partial F(\Ds \vu) \ \Leftrightarrow \ 
\Ds \vu \in \partial F^*(\mathbb{S}). 
\] 
Note that the standard linear Newton's rheological law 
\begin{equation} \label{P8}
\mathbb{S} = \mu \Ds \vu + \lambda \Div \vu \mathbb{I}
\end{equation}
corresponds to 
\[
F(\Ds \vu) = \frac{\mu}{2} |\Ds \vu|^2 + \frac{\lambda}{2} |\Div \vu|^2,\ \mu > 0, \ \lambda + \frac{2}{d} \mu \geq 0.
\]
The resulting problem is the standard Navier--Stokes system governing the motion of a linearly viscous compressible barotropic fluid. 

The iconic example of the pressure--density relation is the isentropic state equation, 
\begin{equation} \label{P9}
p(\vr) = a \vr^\gamma,\ a > 0, \ \gamma > 1.
\end{equation}
In this case, the Navier--Stokes system \eqref{P1}, \eqref{P2}, \eqref{P8}, \eqref{P9} admits global in time weak solutions 
for $\gamma > \frac{d}{2}$, see Lions \cite{LI4} and  \cite{EF70},  for the homogeneous boundary conditions, and 
{\cite{Girinon},} \cite{ChJiNo}, {\cite{KwNo}} for general inflow--outflow boundary conditions. The existence of global weak solutions in the 
case $d=2$, $\gamma = 1$ was proved by Plotnikov and Vaigant \cite{PloWei}.  

Much less is known for a general non--linear dependence of viscosity on the velocity gradient. 
To the best of our knowledge, 
there are only two large--time existence results in the class of weak solutions in the multidimensional case: {\bf (i)} Mamontov \cite{MAM1}, \cite{MAM2} considered the case 
of exponentially growing viscosity coefficients and linear pressure $p(\vr) = a \vr$; {\bf (ii)} \cite{FeLiMa}, where the bulk viscosity 
$\lambda = \lambda(|\Div \vu|)$ becomes singular for a finite value of $|\Div \vu|$.

Motivated by \cite{AbbFei2}, \cite{BreFeiHof19}, we introduce the concept of \emph{dissipative solution} for problem \eqref{P1}--
\eqref{P7}. The dissipative solution satisfies the system of equations  
\begin{equation} \label{P10}
\begin{split}
\partial_t \vr + \Div (\vr \vu) &= 0,\\
\partial_t (\vr \vu) + \Div (\vr \vu \otimes \vu) + \Grad p(\vr) &=\Div [ \Se].
\end{split}
\end{equation}
The effective stress acting on the fluid can be decomposed as 
\begin{equation} \label{P11}
\Se = \mathbb{S} - \mathfrak{R},
\end{equation}
where the ``turbulent'' component $\mathfrak{R}$ 
is called \emph{Reynolds stress} and it is a positively semi--definite tensor. The problem is supplemented with the energy inequality 
\begin{equation} \label{P12}
\begin{split}
\frac{{\rm d}}{{\rm d}t} &\intO{\left[ \frac{1}{2} \vr |\vu - \vu_B|^2 + P(\vr) + \mathfrak{E} \right] }  \\&+ 
\intO{ \Big[ F(\Ds \vu) + F^* (\Se + \mathfrak{R}) \Big] } \dt  
+\int_{\partial \Omega} P(\vr)  \vu_B \cdot \vc{n} \ \D S_x \dt
\\
\leq
&- 
\intO{ \Big[ p(\vr)\mathbb{I} + \vr \vu \otimes \vu \Big]: \Grad \vu_B   } \dt + \intO{ {\vr} \vu  \cdot \Big(\vu_B \cdot \Grad \vu_B  \Big) } 
\dt + \int_{\Ov{\Omega}} \Se : \Grad \vu_B \dxdt.
\end{split}
\end{equation}
Here the symbol $\mathfrak{E}$ 
denotes the energy \emph{dissipation defect} directly related to the Reynolds stress $\mathfrak{R}$,
more specifically, $\mathfrak{E} \approx {\rm tr}[\mathfrak{R}]$. The velocity  
$\vu_B$ is a suitable extension 
of the boundary velocity to $\Ov{\Omega}$.
The quantities $\vr$, $\vu$ are called 
\emph{dissipative solution} if they satisfy \eqref{P10}--\eqref{P12} in the sense of distributions
for suitable $\Se$, $\mathfrak{R}$, and $\mathfrak{E}$. It is easy to show that if 
\begin{equation} \label{class}
\mathfrak{R} = \mathfrak{E} = 0,\ \mathbb{S} \in \partial F(\Ds \vu),
\end{equation}
then $[\vr, \vu]$ is the standard weak solution satisfying a variant of the energy inequality.

The exact definition of dissipative solution is given in Section \ref{W}. Then 
we plan to address the following issues:
\begin{itemize}
\item {\bf Existence.} 
In Section \ref{E}, we show that 
problem \eqref{P10}--\eqref{P12} admits global--in--time solutions for any finite energy initial data. We consider a general 
bounded Lipschitz domain $\Omega$ in view of possible applications in the analysis of numerical schemes.
\item {\bf Compatibility.} 
In Section \ref{C}, we show that dissipative solutions
of \eqref{P10}--\eqref{P12} that are continuously differentiable are classical solutions of \eqref{P1}--\eqref{P7}. In particular, 
\eqref{class} holds.
\item {\bf Weak--strong uniqueness.} A solution of \eqref{P10}--\eqref{P12} coincides with the strong solution of  \eqref{P1}--\eqref{P7} 
emanating from the same initial data as long as the latter exists, see Sections, \ref{RR}, \ref{WU}.

\end{itemize}

\section{Weak formulation}
\label{W}

We consider a bounded Lipschitz domain $\Omega \subset R^d$, $d = 2,3$. 
We suppose that the boundary conditions for the velocity are given by $\vu_B \in C^1(R^d; R^d)$.
The boundary $\partial \Omega$ is decomposed as 
\[
\partial \Omega = \Gamma_{\rm in} \cup \Gamma_{\rm out}, 
\ \Gamma_{\rm in} = \left\{ x \in \partial \Omega \ \Big|\
\ \mbox{the outer normal}\ \vc{n}(x) \ \mbox{exists, and}\ \vu_B(x) \cdot \vc{n}(x) < 0 \right\}. 
\]

\subsection{Structural restrictions on the constitutive equations}

We start by introducing the structural restrictions imposed on the 
pressure--density equation of state (EOS), and the specific form of the viscous stress.

\subsubsection{Pressure--density equation of state}

An iconic example of a barotropic EOS is the isentropic pressure--density relation 
\begin{equation}\label{S1}
p(\vr) = a \vr^\gamma,\ a > 0, \ \gamma > 1\ \mbox{yielding the pressure potential}\ P(\vr) = \frac{a}{\gamma - 1} \vr^\gamma.
\end{equation}
Here we allow for more general EOS retaining the essential features of \eqref{S1}:
\begin{equation} \label{S2}
\begin{split}
&p \in C[0, \infty) \cap C^2(0, \infty),\ 
p(0) = 0, \ p'(\vr) > 0 \ \mbox{for}\ \vr > 0; \\
&\mbox{the pressure potential}\ P 
\ \mbox{determined by}\ P'(\vr) \vr - P(\vr) = p(\vr)\ \mbox{satisfies}\
P(0) = 0,\\ &\mbox{and}\ P - \underline{a} p,\ \Ov{a} p - P\ \mbox{are convex functions 
for certain constants}\ \underline{a} > 0, \ \Ov{a} > 0.
\end{split}
\end{equation} 
As a matter of fact, certain hypotheses on the pressure can be relaxed. 
We shall discuss this issue in the concluding Section \ref{D}. 
It is easy to check that any $P$ satisfying \eqref{S2} possesses certain coercivity 
similar to \eqref{S1}. More specifically,
\begin{equation} \label{S2a}
P(\vr) \geq a \vr^\gamma \ \mbox{for certain}\ a > 0,\ \gamma > 1 \ \mbox{and all}\ \vr \geq 1.
\end{equation}
Indeed as $\Ov{a}p - P$ is a convex function and $P$ is strictly convex, we get 
\[
\Ov{a} p''(\vr) \geq P''(\vr) = \frac{p'(\vr)}{\vr},\ \vr > 0.
\]
This yields 
\[
\log'(p'(\vr)) \geq \log' \left(\vr^{\frac{1}{\Ov{a}}} \right) \ \Rightarrow \ 
p'(\vr) \geq \vr^{\frac{1}{\Ov{a}}} \ \mbox{for all}\ \vr \ \mbox{large enough;}
\] 
whence \eqref{S2a} holds for $\gamma = 1 + \frac{1}{\Ov{a}}$.

\subsubsection{Viscous stress}

The relation between the symmetric velocity gradient $\Ds \vu$ and the viscous stress $\mathbb{S}$ is determined by the choice 
of the potential $F$. We suppose that 
\begin{equation} \label{S3}
F: R^{d \times d}_{\rm sym} \to [0, \infty) \ \mbox{is a (proper) convex function},\ F(0) = 0,
\end{equation}
enjoying certain coercivity properties to render the symmetric gradient $\Ds \vu$, or at least 
its traceless part $\Ds \vu - \frac{1}{d} \Div \vu \mathbb{I}$, integrable. To make the last statement more specific, 
we introduce the class of Young functions $A$:
\begin{itemize} 
\item $A: [0, \infty) \to [0, \infty)$ convex, 
\item $A$ increasing, 
\item $A(0) = 0$.
\end{itemize}
Moreover, we shall say that $A$
satisfies $\Delta^2_2$--condition if there exist constants $a_1 > 2$, $a_2 < \infty$ such that 
\begin{equation} \label{S4}
a_1 A(z) \leq A(2z) \leq a_2 A(z) \ \mbox{for any}\ z \in [0, \infty).
\end{equation}

In addition to \eqref{S3}, we suppose that for any $R > 0$, there exists a Young function $A_R$ satisfying the 
$\Delta^2_2$--condition \eqref{S4} and such that 
\begin{equation} \label{S5}
F(\mathbb{D} + \mathbb{Q}) - F(\mathbb{D}) - \mathbb{S}: \mathbb{Q} \geq 
A_R \left( \left| \mathbb{Q} - \frac{1}{d} {\rm tr}[\mathbb{Q}] \mathbb{I} \right| \right)
\end{equation}
for all $\mathbb{D}, \mathbb{S}, \mathbb{Q} \in R^{d \times d}_{\rm sym}$ such that 
\[
|\mathbb{D}| \leq R,\ \mathbb{S} \in \partial {F}(\mathbb{D}).
\]
{Inequality (\ref{S5}) implies, among others, that $F^*(\mathbb{S})\ge 0$ for all $\mathbb{S} \in \partial {F}(\mathbb{D})$ and $\mathbb{D} \in R^{d \times d}_{\rm sym}$. This observation will be used later in order
to derive estimates on $\mathbb{S}$ from the energy inequality.}

Note that the standard Newtonian rheological law with the associated quadratic potential 
\[
F(\mathbb{D}) = \mu \left|\mathbb{D} - \frac{1}{d} {\rm tr}[\mathbb{D}] \right|^2 + \eta \left| {\rm tr}[\mathbb{D}] \right|^2, 
\ \mu > 0, \ \eta \geq 0,
\]
satisfies \eqref{S5} with 
\[
A_R(z) = A(z) = \frac{\mu}{2}z^2.
\]

For a more general $p-$potential
$$F(\mathbb{D})=\frac{\mu_0}{p}(\mu_1 + |\mathbb{D}_0|^2)^{\frac{p}{2}} \mbox{ with } \mu_0>0, \ \mu_1 \geq 0, \ p>1, $$
noticing that 
$$ F(\mathbb{A}) - F(\mathbb{B}) - \partial F(\mathbb{B}) : (\mathbb{A-B}) = \int_0^1 (\partial F(\mathbb{B}+ t(\mathbb{A-B})) - \partial F(\mathbb{B})) : (\mathbb{A-B}) \dt $$
and making use of \cite[Lemma 6.3]{DiEbRu} we get 
$$ F(\mathbb{A}) - F(\mathbb{B}) - \partial F(\mathbb{B}) : (\mathbb{A-B}) \geq \begin{cases} 
  \min \{ \frac{\mu_0}{2}(\mu_1 +3R)^{p-2} |\mathbb{A-B}|^2,   \frac{\mu_0}{2}(\mu_1 +3R)^{p-2}|\mathbb{A-B}|^p\} &\mbox{ if } p<2,\\
 \frac{\mu_0}{p} |\mathbb{A-B}|^p  &\mbox{ if } p\geq 2,
\end{cases}$$
for any $\mathbb{A}, \mathbb{B}$ such that $|\mathbb{B}|\leq R$. 
Thus relation \eqref{S5} is satisfied with 
 $$A_R(z) = \begin{cases}
 \frac{\mu_0}{2}(\mu_1 +3R)^{p-2} z^2 &\mbox{ if } p<2, \ 0 \leq z \leq 1,\\
 \frac{\mu_0}{2}(\mu_1 +3R)^{p-2}(1-\alpha) z^p + \alpha z -\beta &\mbox{ if } p<2, \ z \geq 1,\\
 \frac{\mu_0}{p} z^p &\mbox{ if } p\geq2,
 \end{cases} $$
where $0<\alpha<1, \ \beta>0$ are constants fixed in order to make $A_R$ continuous and convex. 

Finally, it follows from \eqref{S5} that there exist $\mu > 0$ and $q > 1$ such that
\begin{equation} \label{S5a}
F(\mathbb{D}) \geq \mu \left| \mathbb{D} - \frac{1}{d} {\rm tr}[\mathbb{D}] \mathbb{I} \right|^q 
\ \mbox{for all}\ |\mathbb{D}| > 1.
\end{equation} 
Indeed it is enough to consider $\mathbb{D} = \mathbb{S} = 0$ and $R=1$ in \eqref{S5}. In view of 
\eqref{S4}, the function $A_1$ possesses the desired $q-$growth for large values of its arguments.

\subsection{Weak formulation of the field equations}

We introduce the weak formulation of the conservation laws \eqref{P1}, \eqref{P2}.

\subsubsection{Equation of continuity - mass conservation}

The density $\vr \geq 0$ is non--negative and belongs to the class  
\[
\vr \in C_{\rm weak}([0,T]; L^\gamma_{\rm weak}(\Omega)) \cap L^\gamma(0,T; L^\gamma(\partial \Omega; |\vu_B \cdot \vc{n}|
{\rm d}S_x)),\ 
\gamma > 1,
\]
the momentum satisfies
\[
\vr \vu \in L^\infty(0,T; L^{\frac{2 \gamma}{\gamma + 1}}(\Omega; R^d)),\ \vr \geq  0.
\]
The integral identity 
\begin{equation} \label{W1}
\begin{split}
\left[ \intO{ \vr \varphi } \right]_{t = 0}^{t = \tau} &+ 
\int_0^\tau \int_{\Gamma_{\rm out}} \varphi \vr \vu_B \cdot \vc{n} \ \D \ S_x 
+ 
\int_0^\tau \int_{\Gamma_{\rm in}} \varphi \vr_B \vu_B \cdot \vc{n} \ \D \ S_x\\ &= 
\int_0^\tau \intO{ \Big[ \vr \partial_t \varphi + \vr \vu \cdot \Grad \varphi \Big] } \dt 
\end{split}
\end{equation}
holds
for any $0 \leq \tau \leq T$, and any test function $\varphi \in C^1([0,T] \times \Ov{\Omega})$,
\begin{equation} \label{W2}
\vr(0, \cdot) = \vr_0.
\end{equation}

\subsubsection{Momentum balance}

Let the symbol 
$\mathcal{M}^+(\Ov{\Omega}; R^{d \times d}_{\rm sym})$ denote the set of all positively semi--definite tensor valued measures on $
\Ov{\Omega}$.
We suppose there exist 
\[
\mathbb{S} \in L^1((0,T) \times \Omega; R^{d \times d}_{\rm sym}),\ \mathfrak{R} \in L^\infty(0,T; \mathcal{M}^+(\Ov{\Omega}; R^{d \times d}_{\rm sym})),
\]
such that the integral identity
\begin{equation} \label{W3}
\begin{split}
\left[ \intO{ \vr \vu \cdot \bfphi } \right]_{t=0}^{t = \tau} &= 
\int_0^\tau \intO{ \Big[ \vr \vu \cdot \partial_t \bfphi + \vr \vu \otimes \vu : \Grad \bfphi 
+ p(\vr) \Div \bfphi - \mathbb{S} : \Grad \bfphi \Big] }\\
&+ \int_0^\tau \int_{{\Omega}} \Grad \bfphi : \D \ \mathfrak{R}(t) \ \dt
\end{split}
\end{equation}
holds for any $0 \leq \tau \leq T$ and any test function $\bfphi \in C^1([0,T] \times \Ov{\Omega}; R^d)$, $\bfphi|_{\partial \Omega} = 0$,
\begin{equation} \label{W4}
\vr \vu(0, \cdot) = \vm_0
\end{equation}
Here we assume that all quantities appearing in \eqref{W3} are at least integrable in $(0,T) \times \Omega$. In accordance with 
\eqref{P10}, we may set 
\[
\Se = \mathbb{S} - \mathfrak{R}.
\]

\subsection{Energy inequality and defect compatibility condition}

A proper form of the energy balance \eqref{P3} is a cornerstone of the subsequent analysis. We first introduce the 
energy defect measure 
\[
\mathfrak{E} \in L^\infty(0,T; \mathcal{M}^+ (\Ov{\Omega})). 
\]
The energy inequality reads 
\begin{equation} \label{W5}
\begin{split}
&\left[ \intO{\left[ \frac{1}{2} \vr |\vu - \vu_B|^2 + P(\vr) \right] } \right]_{t = 0}^{ t = \tau} + 
\int_0^\tau \intO{ \Big[ F(\Ds \vu) + F^* (\mathbb{S}) \Big] } \dt\\  
&+\int_0^\tau \int_{\Gamma_{\rm out}} P(\vr)  \vu_B \cdot \vc{n} \ \D S_x \dt +
\int_0^\tau \int_{\Gamma_{\rm in}} P(\vr_B)  \vu_B \cdot \vc{n} \ \D S_x \dt
+ \int_{\Ov{\Omega}}  \D \ \mathfrak{E} (\tau) \\	
\leq
&- 
\int_0^\tau \intO{ \left[ \vr \vu \otimes \vu + p(\vr) \mathbb{I} \right]  :  \Grad \vu_B } \dt + \int_0^\tau \intO{ {\vr} \vu  \cdot \vu_B \cdot \Grad \vu_B  } 
\dt + \int_0^\tau \intO{ \mathbb{S} : \Grad \vu_B } \dt \\ &- 
\int_0^\tau \int_{\Ov{\Omega}} \Grad \vu_B : \D \ \mathfrak{R}(t) \dt.  
\end{split}
\end{equation}
for a.a. $0 \leq \tau \leq T$. 

Finally, we suppose a compatibility conditions between
the energy defect $\mathfrak{E}$ dominates the Reynolds stress $\mathfrak{R}$, specifically, 
\begin{equation} \label{W6}
\underline{d} \mathfrak{E} \leq {\rm tr}[\mathfrak{R}] \leq \Ov{d} \mathfrak{E},\ 
\ \mbox{for certain constants}\ 0 < \underline{d} \leq \Ov{d}.
\end{equation}
This property is absolutely crucial for the weak--strong uniqueness principle stated in Section \ref{WU} below. 

\begin{Definition}[{\bf Dissipative solution}] \label{WD1}

Let $\Omega \subset R^d$, $d=2,3$ be a bounded Lipschitz domain. The quantity $[\vr, \vu]$ is called 
\emph{dissipative solution} of the problem \eqref{P1}--\eqref{P7} if:
\begin{itemize}
\item
\[
\vr : [0,T] \times \Ov{\Omega} \to R,\ \vr \in C_{{\rm weak}}([0,T]; L^\gamma (\Omega)) 
\cap L^\gamma(0,T; L^\gamma (\partial \Omega, |{\vu_B} \cdot \vc{n}| \D  S_x)), \ \gamma > 1,
\]
\[
\vu \in L^q (0,T; W^{1,q}(\Omega; R^d)),\ (\vu - \vu_B) \in L^q(0,T; W^{1,q}_0(\Omega; R^d)),\ q > 1,  
\]
\[
\vm \equiv \vr \vu \in C_{{\rm weak}}([0,T]; L^{\frac{2 \gamma}{\gamma + 1}}(\Omega; R^d)).
\]

\item There exist 
\[
\mathbb{S} \in L^1((0,T) \times \Omega; R^{d \times d}_{\rm sym}),\ 
\mathfrak{E} \in L^\infty(0,T; \mathcal{M}^+(\Ov{\Omega})),\ 
\mathfrak{R} \in L^\infty(0,T; \mathcal{M}^+(\Ov{\Omega}; R^{d \times d}_{\rm sym}))
\]
such that the relations \eqref{W1}--\eqref{W6} are satisfied for any $0 \leq \tau \leq T$.

\end{itemize}

\end{Definition}

\begin{Remark} \label{CR1}

In view of \eqref{W6}, one can always consider 
\begin{equation} \label{DD1}
\mathfrak{E} \equiv \frac{1}{\Ov{d}} {\rm tr}[\mathfrak{R}]; 
\end{equation}
whence, strictly speaking, the energy defect $\mathfrak{E}$ can be completely omitted in the definition, see 
the discussion in Section \ref{D} for details.

\end{Remark}

\begin{Remark} \label{CR2}

As we shall see in the existence proof below, the dissipative solutions can be constructed in such a way that the constant 
$\Ov{d}$ depends solely on the dimension $d$ and the structural constants $\underline{a}$, $\Ov{a}$ appearing in \eqref{S2}.

\end{Remark}

\section{Existence}
\label{E}

Our first goal is to show that the dissipative solutions exist globally in time for any finite energy initial data. 
The proof is based on a multilevel approximation scheme that shares certain features with the approximation of the 
compressible Navier--Stokes in \cite{EF70}. First, we introduce a sequence of finite--dimensional spaces 
$X_n \subset L^2(\Omega; R^d)$,
\[
X_n = {\rm span} \left\{ \vc{w}_i\ \Big|\ \vc{w}_i \in \DC(\Omega; R^d),\ i = 1,\dots, n \right\}
\]   
Without loss of generality, we may assume that $\vc{w}_i$ are orthonormal with respect to the standard scalar product in $L^2$.

Next, we regularize the convex potential $F$ to make it continuously differentiable. This may be done via convolution with a family 
of regularizing kernels 
\[
F_\delta(\mathbb{D}) = \int_{R^{d \times d}_{\rm sym}} \xi_\delta ( | \mathbb{D} - \mathbb{Z} | ) F(\mathbb{Z}) \ {\rm d} \mathbb{Z}
- \int_{R^{d \times d}_{\rm sym}} \xi_\delta ( | \mathbb{Z} | ) F(\mathbb{Z}) \ {\rm d} \mathbb{Z}.
\]
It is easy to check  that $F_\delta$ are convex, non--negative, infinitely differentiable, $F_\delta (0) = 0$, and satisfy \eqref{S5a}, specifically
\begin{equation} \label{F1a}
F_\delta (\mathbb{D}) \geq \nu \left| \mathbb{D} - \frac{1}{d} {\rm tr}[\mathbb{D}] \mathbb{I} \right|^q 
\ \mbox{for all}\ |\mathbb{D}| > 1,
\end{equation}
with $q > 1$, $\nu > 0$ independent of $\delta \searrow 0$.

\subsection{First level approximation}
 
To begin, we suppose that the initial and boundary data are smooth. Specifically, 
\begin{itemize}
 \item [{\bf (A1)}] $\vu_B \in C^{1}_c(R^d; R^d)$; 
\item [{\bf (A2)}] $\vr_{0} \in C^{1}(R^d)$, $\vr_B \in C^1(\partial \Omega)$. 
\end{itemize}

\subsubsection{Artificial viscosity approximation of the equation of continuity}

Following
Crippa, Donadello, Spinolo \cite{CrDoSp} we use a parabolic approximation of the equation of continuity,
\begin{equation} \label{E1}
\partial_t \vr + \Div (\vr \vu ) = \ep \Del \vr \ \mbox{in}\ (0,T) \times \Omega,\ \ep > 0,
\end{equation}
supplemented with the boundary conditions 
\begin{equation} \label{E2}
\ep \Grad \vr \cdot \vc{n} + (\vr_B - \vr) [\vu_B \cdot \vc{n}]^- = 0 \ \mbox{in}\ [0,T] \times \partial \Omega,
\end{equation}
and the initial condition
\begin{equation} \label{E3}
\vr(0, \cdot) = \vr_0.
\end{equation}
Here, $\vu = \vv + \vu_B$, with $\vv \in C([0,T]; X_n)$, in particular, $\vu|_{\partial \Omega} = \vu_B$. 
Note that for given $\vu$, $\vr_B$, $\vu_B$, this is a linear problem for the unknown $\vr$. 

As $\Omega$ is merely Lipschitz, the usual parabolic estimates fail at the level of the spatial derivatives and we 
are forced to use the weak formulation: 
\begin{equation} \label{E4}
\begin{split}
\left[ \intO{ \vr \varphi } \right]_{t = 0}^{t = \tau}&= 
\int_0^\tau \intO{ \left[ \vr \partial_t \varphi + \vr \vu \cdot \Grad \varphi - \ep \Grad \vr \cdot \Grad \varphi \right] }
\dt \\ 
&- \int_0^\tau \int_{\partial \Omega} \varphi \vr \vu_B \cdot \vc{n} \ {\rm d} S_x \dt + 
\int_0^\tau \int_{\partial \Omega} \varphi (\vr - \vr_B) [\vu_B \cdot \vc{n}]^{-}  \ {\rm d}S_x \dt,\ 
\vr(0, \cdot) = \vr_0, 
\end{split}
\end{equation}
for any test function
\[
\varphi \in L^2(0,T; W^{1,2}(\Omega)),\ \partial_t \varphi \in L^1(0,T; L^2(\Omega)).
\]

\begin{Lemma} [{\bf Crippa et al \cite[Lemma 3.2]{CrDoSp}} ] \label{EL1}
Let $\Omega \subset R^d$ be a bounded Lipschitz domain.

Given $\vu = \vv + \vu_B$, $\vv \in C([0,T]; X_n)$, the initial--boundary value problem \eqref{E1}--\eqref{E3} 
admits a weak solution $\vr$ {specified in (\ref{E4})}, unique in the class 
\[
\vr \in L^2(0,T; W^{1,2}(\Omega)) \cap C([0,T]; L^2(\Omega)).
\]
The norm in the aforementioned spaces is bounded only in terms of the data $\vr_B$, $\vr_0$, $\vu_B$ and 
\[
\sup_{t \in (0,T)}
\| \vv(t, \cdot) \|_{X_n}.
\]

\end{Lemma}

Next, we report another result of Crippa et al. \cite[Lemma 3.4]{CrDoSp}. 

\begin{Lemma}[{\bf Maximum principle}] \label{EL2}

Under the hypotheses of Lemma \eqref{EL1}, the solution $\vr$ satisfies 
\[
\| \vr \|_{L^\infty((0,\tau) \times \Omega)} \leq 
\max \left\{ \| \vr_0 \|_{L^\infty(\Omega)}; \| \vr_B \|_{L^\infty((0,T) \times \Gamma_{in})}; 
\| \vu_B \|_{L^\infty((0,T) \times \Omega)} \right\} \exp \left( \tau \| \Div \vu \|_{L^\infty((0,\tau) \times \Omega)} \right).
\] 

\end{Lemma}

Next, we perform renormalization of equation \eqref{E1}. {In view of future applications, in particular the 
strong minimum principle, it is convenient to rewrite the integral identity \eqref{E4} in terms of a new variable 
$r = r(t,x) \equiv \vr (t,x) - \chi (t)$, where $\chi \in W^{1, \infty} (0,T)$. After a straightforward manipulation, we obtain:}
\begin{equation} \label{E4bis}
\begin{split}
\left[ \intO{ r \varphi } \right]_{t = 0}^{t = \tau}&= 
\int_0^\tau \intO{ \left[ r \partial_t \varphi - \Big( \partial_t \chi  + \chi \Div \vu \Big) 
\varphi + r \vu \cdot \Grad \varphi - \ep \Grad r \cdot \Grad \varphi \right] }
\dt \\ 
&- \int_0^\tau \int_{\partial \Omega} \varphi r  \vu_B \cdot \vc{n} \ {\rm d} S_x \dt + 
\int_0^\tau \int_{\partial \Omega} \varphi ((r + \chi) - \vr_B) [\vu_B \cdot \vc{n}]^{-}  \ {\rm d}S_x \dt,\\ 
r(0, \cdot) &= \vr_0 - \chi(0), 
\end{split}
\end{equation}
{for any test function}
\[
\varphi \in L^2(0,T; W^{1,2}(\Omega)),\ \partial_t \varphi \in L^1(0,T; L^2(\Omega)).
\]

\begin{Lemma}[{{\bf Renormalization}}] \label{EL3}

Under the conditions of Lemma \ref{EL2},
let $B \in C^2(R)$, $\chi \in W^{1, \infty}(0,T)$, and $r = \vr - \chi$.

Then, the (integrated) renormalized equation
\[
\begin{split}
\left[ \intO{ B(r) } \right]_{t = 0}^{t = \tau} 
= &- \int_{0}^{\tau} \intO{ \Div (r \vu) B'(r) } \dt - \ep \int_{0}^{\tau} \intO{ |\Grad r|^2  B''(r) } \dt \\
&- \intO{ \Big( \partial_t \chi + \chi \Div \vu \Big) B'(r) } 
\\ & + \int_{0}^{\tau} \int_{\partial \Omega} B'(r) ((r + \chi) - \vr_B) [\vu_B \cdot \vc{n}]^- \ {\rm d} S_x \ \dt
\end{split}
\]
holds for any $0 \leq \tau \leq T$.

\end{Lemma}

\begin{proof}

To begin, in accordance the maximum principle stated in Lemma \ref{EL2}, we may assume $B \in C^2_c(R)$. Moreover, 
as $\vu = \vv + \vu_B$, we have  
\[
\Div (\vr \vu) = \Grad \vr (\vv + \vu_B) + \varrho \Div (\vv + \vu_B).
\]
Thus, in view of the bounds on $\vr$ obtained in Lemma \ref{EL1}, we have 
\[
\| \Div (\vr \vu) \|_{L^2((0,T) \times \Omega)} \leq c 
\]
in terms of the data only.

As 
\[
\partial_t \vr \in L^2(0,T; (W^{1,2})^*(\Omega)), \ B'(r) \in L^2(0,T; W^{1,2}(\Omega)) 
\]
we deduce from the weak formulation \eqref{E4bis} that 
\[
\begin{split}
\left< \partial_t r, B'(r) \right>_{[W^{1,2}]^*; W^{1,2}} = &- \intO{ \Div (r \vu) B'(r) } - \ep \intO{ |\Grad r|^2  B''(r) }\\
&- \intO{ \Big( \partial_t \chi + \chi \Div \vu \Big) B'(r) } 
\\
&+ \int_{\partial \Omega} B'(r) ((r + \chi) - \vr_B) [\vu_B \cdot \vc{n}]^- \ {\rm d} S_x
\end{split}
\]
Integrating over the interval $[\tau_1, \tau_2]$ we get 
\[
\begin{split}
\int_{\tau_1}^{\tau_2} &\left< \partial_t r, B'(r) \right>_{[W^{1,2}]^*; W^{1,2}} \dt\\ 
= &- \int_{\tau_1}^{\tau_2} \intO{ \Div (r \vu) B'(r) } \dt - \ep \int_{\tau_1}^{\tau_2} \intO{ |\Grad r|^2  B''(r) } \dt \\
&- \int_{\tau_1}^{\tau_2} \intO{ \Big( \partial_t \chi + \chi \Div \vu \Big) B'(r) } \dt 
\\ & + \int_{\tau_1}^{\tau_2} \int_{\partial \Omega} B'(r) ((r + \chi)  - \vr_B) [\vu_B \cdot \vc{n}]^- \ {\rm d} S_x \ \dt
\end{split}
\]

Finally, using the standard temporal regularization via a family of $t-$dependent convolution kernels, we find a sequence of functions 
\[
\varrho_n \to \vr \ \mbox{in}\ L^2((\tau_1, \tau_2); W^{1,2}(\Omega)),\ 
\varrho_n \in C^1([\tau_1, \tau_2]; W^{1,2}(\Omega)), \ \partial_t \vr_n \to \partial_t \vr  
\ \mbox{in}\ L^2(\tau_1, \tau_2; [W^{1,2}]^*(\Omega)) 
\]
for any $0 < \tau_1 < \tau_2$, and, as $\vr \in C([0,T]; L^2(\Omega)) \cap L^\infty((0,T) \times \Omega)$,
\[
B(\vr_n - \chi )(\tau) \to B(\vr(\tau) - \chi(\tau) ) \ \mbox{for any}\ \tau \in (0,T).
\]
Thus we obtain 
\[
\begin{split}
\left[ \intO{ B(r) } \right]_{t = \tau_1}^{t = \tau_2} &= 
\lim_{n \to \infty} \left[ \intO{ B(\vr_n - \chi ) } \right]_{t = \tau_1}^{t = \tau_2}\\
&= \lim_{n \to \infty} \int_{\tau_1}^{\tau_2} \intO{ B'(\vr_n - \chi ) \partial_t (\vr_n - \chi) } \dt 
= \int_{\tau_1}^{\tau_2} \left< \partial_t r, B'(r) \right> \dt
\end{split}
\]
for any $0 < \tau_1 < \tau_2 < T$, which yields the desired conclusion. 

\end{proof}

Using Lemma \ref{EL3} we obtain strict positivity of $\vr$ on condition that $\vr_B$, $\vr_0$ enjoy the same property. 

\begin{Corollary} \label{EC1}

Under the hypotheses of Lemma \ref{EL1}, we have 
\[
{\rm ess} \inf_{t,x} \vr(t,x) \geq  
\min \left\{ \min_\Omega \vr_0 ; \min_{\Gamma_{\rm in}} \vr_B \right\} 
\exp \left( -T \| \Div \vu \|_{L^\infty((0,T) \times \Omega)} \right).
\]

\end{Corollary}

{Indeed it is enough to apply Lemma \ref{EL1} to} 
\[
\chi(\tau) = \underline{\vr} \exp \left( - \int_0^\tau \| \Div \vu(t, \cdot) \|_{L^\infty(\Omega)} \dt \right),\ 
\underline{\vr} = \min \left\{ \min_\Omega \vr_0 ; \min_{\Gamma_{\rm in}} \vr_B \right\}
\]
{and}
\[
B(r) = - r^- = \left\{ \begin{array}{l} - r \ \mbox{if}\ r \leq 0,\\ 
0 \ \mbox{if} \ r > 0 \end{array} \right.
\]
{to deduce}
\[
\intO{ B(r)(\tau) } = 0 \ \mbox{for a.a.}\ \tau \in (0,T) 
\ \Rightarrow \ r = \vr - \chi \geq 0.
\]  

Finally, we observe that
in the present setting, specifically for smooth and \emph{time independent} vector fields $\vr_B$, $\vu_B$, the weak solution enjoys more regularity than in Crippa et al. \cite{CrDoSp}. In particular, we have the estimates
\begin{equation} \label{E4a}
\| \partial_t \vr \|_{L^2((0,T) \times \Omega)} + 
\ep {\rm ess} \ \sup_{\tau \in (0,T)} \left\| \Grad \vr (\tau, \cdot) \right\|^2_{L^2(\Omega)} \leq c,
\end{equation} 
with the constant depending only on the data. The estimate \eqref{E4a} can be first obtained formally via multiplying the equation 
\eqref{E1} on $\partial_t \vr$:
\[
\intO{ |\partial_t \vr|^2 } + \intO{ \Div (\vr \vu ) \partial_t \vr} = \ep \intO{\Del \vr \partial_t \vr } = 
- \frac{\ep}{2} \frac{{\rm d}}{{\rm d}t} \intO{ |\Grad \vr |^2 } + 
\ep \int_{\partial \Omega} \Grad \vr \cdot \vc{n} \partial_t \vr {\rm d} S_x,
\]
where, in accordance with the boundary condition \eqref{E2}, 
\[
\ep \int_{\partial \Omega} \Grad \vr \cdot \vc{n} \partial_t \vr {\rm d} S_x = 
\int_{\partial \Omega} \partial_t \vr (\vr - \vr_B) [\vu_B \cdot \vc{n}]^- {\rm d} S_x = \frac{1}{2} 
\frac{{\rm d}}{{\rm d}t} \int_{\partial \Omega}(\vr - \vr_B)^2  [\vu_B \cdot \vc{n}]^- {\rm d} S_x.
\] 
Consequently, the desired estimate \eqref{E4a} follows by integrating the above relation over time. As pointed out 
in Crippa et al. \cite{CrDoSp}, the (unique) weak solution $\vr$ can be constructed by means of Faedo--Galerkin approximation. 
The latter being compatible with multiplication on $\partial_t \vr$, the above argument can be performed on the approximation 
and thus transfered to the limit via lower semi--continuity of the associated norms.

\subsubsection{Galerkin approximation of the momentum balance}

We look for approximate velocity field in the form 
\[
\vu = \vv + \vu_B, \ \vv \in C([0,T]; X_n). 
\]
Accordingly, the
approximate momentum balance reads
\begin{equation} \label{E5}
\begin{split}
\left[ \intO{ \vr \vu \cdot \bfphi } \right]_{t=0}^{t = \tau} &= 
\int_0^\tau \intO{ \Big[ \vr \vu \cdot \partial_t \bfphi + \vr \vu \otimes \vu : \Grad \bfphi 
+ p(\vr) \Div \bfphi - \partial F_\delta (\Ds \vu)  : \Grad \bfphi \Big] }\\
&- \ep \int_0^\tau \intO{ \Grad \vr \cdot \Grad \vu \cdot \bfphi } \dt
\end{split}
\end{equation}
for any $\bfphi \in C^1([0,T]; X_n)$, with the initial condition
\begin{equation} \label{E6}
\vr \vu(0, \cdot) = \vr_0 \vu_0, \ \vu_0 = \vv_0 + \vu_B, \ \vv_0 \in X_n.
\end{equation}

For fixed parameter $n$, $\delta > 0$, $\ep > 0$, the first level approximation are solutions $[\vr, \vu]$ of the problem 
of the parabolic problem \eqref{E1}--\eqref{E3}, and the Galerkin approximation \eqref{E5}, \eqref{E6}. 
The \emph{existence} of the approximate solutions at this level can be proved in the same way as in \cite{ChJiNo}.  Specifically, 
for $\vu = \vu_B + \vv$, $\vv \in C([0,T]; X_n)$, we identify the unique solution $\vr = 
\vr [\vu]$ of \eqref{E1}--\eqref{E3} and plug it as $\vr$ in \eqref{E5}. The unique solution $\vu = \vu[\vr]$ of \eqref{E5} defines a mapping 
\[
\mathcal{T}: \vv \in C([0,T]; X_n) \mapsto  \vr[\vv + \vu_B] = \mathcal{T}[\vc{v}] = (\vu[\vr] - \vu_B)  
\in C([0,T]; X_n).
\]
The first level approximate solutions $\vr = \vr_{\delta, \ep, n}$, $\vu = \vu_{\delta, \ep, n}$ are obtained via a fixed point 
though the mapping $\mathcal{T}$. The exact procedure is detailed in \cite{ChJiNo} and in { \cite{KwNo}}, from where we report the following result.\footnote{ The energy inequality (\ref{E7}) in \cite[Lemma 4.2]{KwNo} is derived under assumption $\Omega\in C^2$. This assumption is needed due to the treatment of the parabolic problem (\ref{E1}--\ref{E3}) via
the classical maximal regularity methods. With Lemmas \ref{EL1}--\ref{EL3} and Corollary \ref{EC1} at hand, the same proof can be carried out
without modifications also in Lipschitz domains.}

\begin{Proposition}[{\bf Approximate solutions, level I}] \label{EP1}

Let $\Omega \subset R^d$, $d=2,3$, be a bounded Lipschitz domain. Suppose that $p = p(\vr)$ and $F = F(\mathbb{D})$ satisfy 
\eqref{S2}, \eqref{S3}--\eqref{S5}. Let the data belong to the class 
\[
\vu_B \in C^{1}_c(R^d; R^d), \ \vr_{0} \in C^{1}(R^d),\ \vr_B \in C^1(\partial \Omega),\ \vr_0, {\vr_B} \geq \underline{\vr} > 0,\ 
\vu_0 = \vv_0 + \vu_B,\ \vv_0 \in X_n.
\]

Then for each fixed $\delta > 0$, $\ep > 0$, $n > 0$, there exists a solution 
$\vr$, $\vu$ of the approximate problem \eqref{E1}--\eqref{E3} and \eqref{E5}, \eqref{E6}. Moreover, the approximate energy inequality 
\begin{equation} \label{E7}
\begin{split}
&\left[ \intO{\left[ \frac{1}{2} \vr |\vu - \vu_B|^2 + P(\vr) \right] } \right]_{t = 0}^{ t = \tau} + 
\int_0^\tau \intO{ \partial F_\delta( \Ds \vu) : \Ds \vu } \dt \\ 
&+\int_0^\tau \int_{\Gamma_{\rm out}} P(\vr)  \vu_B \cdot \vc{n} \ \D S_x \dt\\
&- \int_0^\tau \int_{\Gamma_{\rm in}} \left[ P(\vr_B) - P'(\vr) (\vr_B - \vr) - P(\vr) \right] \vu_B \cdot \vc{n} 
\ {\rm d}S_x \ \dt + \ep \int_0^\tau \intO{ P''(\vr) |\Grad \vr|^2 } \dt
\\
\leq
&- 
\int_0^\tau \intO{ \left[ \vr \vu \otimes \vu + p(\vr) \mathbb{I} \right]  :  \Grad \vu_B } \dt + \int_0^\tau \intO{ {\vr} \vu  \cdot \vu_B \cdot \Grad \vu_B  } 
\dt\\ &+ \int_0^\tau \intO{ \partial F_\delta( \Ds \vu) : \Grad \vu_B } \dt - \int_0^\tau \int_{\Gamma_{\rm in}} P(\vr_B)  \vu_B \cdot \vc{n} \ \D S_x \dt  
\end{split}
\end{equation}
holds for any $0 \leq \tau \leq T$.

\end{Proposition}

\subsection{Second level approximation}

The next step is to let $\delta \to 0$ in the regularization of the potential $F_{\delta}$. This is an easy task as all the necessary bounds from the previous step remain valid. Indeed, in view of the uniform bound \eqref{F1a} and the energy inequality \eqref{E7}, 
we deduce a uniform bound on the traceless part of the symmetric velocity gradient, 
\[
\left\| \Ds \vu - \frac{1}{d} \Div \vu \mathbb{I} \right\|_{L^q((0,T) \times \Omega; R^{d \times d})} \leq c,\ q > 1 
\]
uniformly for $\delta \to 0$. In view of the fact $\vu = \vv + \vu_B$, $\vu|_{\partial \Omega} = 0$, we may use the 
$L^q-$version of Korn's inequality to obtain 
\[
\left\|  \Grad \vu  \right\|_{L^q((0,T) \times \Omega; R^{d \times d})},
\]
which, combined with the standard Poincar\' e inequality, yields the final conclusion 
\begin{equation} \label{E10}
\left\| \vu \right\|_{L^q(0,T; W^{1,q}(\Omega; R^d))} \leq c \ \mbox{for some}\ q > 1.
\end{equation}
At this stage, $n$ is fixed and all norms are equivalent on the finite--dimensional space $X_n$. In particular, 
\[
\left\| \Grad \vu \right\|_{L^\infty((0,T) \times \Omega; R^{d \times d})}
\]
remains bounded uniformly for $\delta \searrow 0$. Therefore it is standard to perform the limit $\delta \to 0$. Accordingly, we
have  
obtained the same conclusion as in Proposition \ref{EP1}, with \eqref{E5} replaced by 
\begin{equation} \label{E11}
\begin{split}
\left[ \intO{ \vr \vu \cdot \bfphi } \right]_{t=0}^{t = \tau} &= 
\int_0^\tau \intO{ \Big[ \vr \vu \cdot \partial_t \bfphi + \vr \vu \otimes \vu : \Grad \bfphi 
+ p(\vr) \Div \bfphi - \mathbb{S}  : \Grad \bfphi \Big] }\\
&- \ep \int_0^\tau \intO{ \Grad \vr \cdot \Grad \vu \cdot \bfphi } \dt,\ 
\mathbb{S} \in L^\infty((0,T) \times \Omega; R^{d \times d}_{\rm sym}),
\end{split}
\end{equation}
for any $\bfphi \in C^1([0,T]; X_n)$, and with the energy inequality in the form 
\begin{equation} \label{E12}
\begin{split}
&\left[ \intO{\left[ \frac{1}{2} \vr |\vu - \vu_B|^2 + P(\vr) \right] } \right]_{t = 0}^{ t = \tau} + 
\int_0^\tau \intO{\left[ F(\Ds \vu) + F^*(\mathbb{S}) \right] } \dt \\ 
&+\int_0^\tau \int_{\Gamma_{\rm out}} P(\vr)  \vu_B \cdot \vc{n} \ \D S_x \dt\\
&- \int_0^\tau \int_{\Gamma_{\rm in}} \left[ P(\vr_B) - P'(\vr) (\vr_B - \vr) - P(\vr) \right] \vu_B \cdot \vc{n} 
\ {\rm d}S_x \ \dt + \ep \int_0^\tau \intO{ P''(\vr) |\Grad \vr|^2 } \dt
\\
\leq
&- 
\int_0^\tau \intO{ \left[ \vr \vu \otimes \vu + p(\vr) \mathbb{I} \right]  :  \Grad \vu_B } \dt + \int_0^\tau \intO{ {\vr} \vu  \cdot \vu_B \cdot \Grad \vu_B  } 
\dt\\ &+ \int_0^\tau \intO{ \mathbb{S} : \Grad \vu_B } \dt - \int_0^\tau \int_{\Gamma_{\rm in}} P(\vr_B)  \vu_B \cdot \vc{n} \ \D S_x \dt  
\end{split}
\end{equation}

\subsection{The third level approximation}

Our next goal is to send $\ep \to 0$ in the viscous approximation \eqref{E1}. This is a bit more delicate than the preceding step as we 
are loosing compactness of the approximate density in the spatial variable. We start by collecting the necessary estimates 
independent of $\ep$. 

Similarly to the preceding section, we have \eqref{E10}, which, as $n$ is still fixed, gives rise to 
\begin{equation} \label{E13}
\| \vu \|_{L^q(0,T; W^{1,\infty}(\Omega))} \leq c,\ q > 1, 
\end{equation}
yielding, in view of Lemma \ref{EL2} and Corollary \ref{EC1}, the uniform bounds on the density 
\begin{equation} \label{E14}
0 < \underline{\vr} \leq \vr(t,x) \leq \Ov{\vr} \ \mbox{for all}\ (t,x) \in [0,T] \times \Ov{\Omega}.
\end{equation} 

Note that at this stage $\vr$ possess a well defined trace on $\partial \Omega$.
This finally implies, by virtue of the energy inequality \eqref{E12}, 
\begin{equation} \label{E15a}
\ep \int_0^T \intO{ |\Grad \vr |^2 } \dt \leq c, 
\end{equation}
and 
\begin{equation} \label{E15}
\sup_{\tau \in [0,T]} \left\| \vu(\tau, \cdot) \right\|_{W^{1, \infty}(\Omega; R^d)} \leq c 
\end{equation}

Now fix $n > 0$ and denote $[ \vre, \vue]$ the approximate solutions constructed in the previous step for each $\ep > 0$. 
In view of the uniform bounds established above, we may assume 
\[
\vre \to \vr \ \mbox{weakly-(*) in}\ L^\infty((0,T) \times \Omega) 
\ \mbox{and weakly in}\ C_{\rm weak}([0,T]; L^r(\Omega)) \ \mbox{for any}\ 1 < r < \infty, 
\] 
passing to a suitable subsequence as the case may be. Note that the second convergence follows form the weak bound on the 
time derivative $\partial_t \vre$ obtained from equation \eqref{E4}. We also have 
\[
\vre \to \vr \ \mbox{weakly-(*) in}\ L^\infty((0,T) \times \partial \Omega; \D S_x).
\]
In addition, the limit density admits the same upper and lower bounds as in \eqref{E14}. 

Similarly, 
\begin{equation} \label{E15b}
\vue \to \vu \ \mbox{weakly-(*) in}\ L^\infty(0,T; W^{1,\infty}(\Omega; R^d)), 
\end{equation}
and 
\[
\vre \vue \to \vc{m} \ \mbox{weakly-(*) in} \ L^\infty((0,T) \times \Omega; R^d)).
\]
Moreover, an abstract version of Arzela--Ascoli theorem yields 
\begin{equation} \label{E15c}
\vm = \vr \vu \ \mbox{a.a. in}\ (0,T) \times \Omega.
\end{equation}

\subsubsection{The limit in the approximate equation of continuity}

Keeping the gradient estimate \eqref{E15} in mind, it is easy to pass to the limit in the regularized continuity 
equation \eqref{E4}:
\begin{equation} \label{E16}
\begin{split}
\left[ \intO{ \vr \varphi } \right]_{t = 0}^{t = \tau}&= 
\int_0^\tau \intO{ \Big[ \vr \partial_t \varphi + \vr \vu \cdot \Grad \varphi  \Big] }
\dt \\ 
&- \int_0^\tau \int_{\Gamma_{\rm out}} \varphi \vr \vu_B \cdot \vc{n} \ {\rm d} S_x \dt - 
\int_0^\tau \int_{\partial \Gamma_{\rm in}} \varphi \vr_B \vu_B \cdot \vc{n}  \ {\rm d}S_x \dt,\ 
\vr(0, \cdot) = \vr_0 
\end{split}
\end{equation}
for any $\varphi \in C^1([0,T] \times \Ov{\Omega})$,
which is a weak formulation of the equation of continuity \eqref{P1}, with the boundary conditions \eqref{P5}, and the initial condition 
\eqref{P7}. Note that the quantity 
\[
\vm = \left\{ \begin{array}{l} \vr \vu_B \cdot \vc{n} \ \mbox{on}\ \Gamma_{\rm out},\\ \\
\vr_B \vu_B \cdot \vc{n} \ \mbox{on}\ \Gamma_{\rm in} \end{array} \right.  
\]
is the normal trace of the divergenceless vector field $[\vr, \vr \vu]$ on the lateral boundary 
$(0,T) \times \partial \Omega$ in the sense of Chen, Torres, and Ziemer \cite{ChToZi}.

\subsubsection{The limit in the approximate momentum equation}

The limit passage in the momentum equation \eqref{E11} is more delicate. First observe that $F^*$ is a superlinear function 
since $F$ is proper convex, ${\rm Dom}[F] = R^{d \times d}_{\rm sym}$. In particular, we may assume 
\[
\mathbb{S}_\ep \to \mathbb{S} \ \mbox{weakly in}\ L^1((0,T) \times \Omega; R^{d \times d}). 
\]
Next, we deduce from \eqref{E11} that 
\[
\partial_t 
\Pi_n[ \vre \vue ] \ \mbox{bounded in}\ L^2(0,T; X_n), 
\]
where $\Pi_n: L^2 \to X_n$ is the associated orthogonal projection. In particular, in view of \eqref{E15b}, \eqref{E15c}, 
we may infer that 
\[
\vre \vue \otimes \vue \to \vr \vu \otimes \vu \ \mbox{weakly-(*) in}\ 
L^\infty((0,T) \times \Omega). 
\]
Consequently, we may let $\ep \to 0$ in \eqref{E11} obtaining 
\begin{equation} \label{E20}
\begin{split}
\left[ \intO{ \vr \vu \cdot \bfphi } \right]_{t=0}^{t = \tau} &= 
\int_0^\tau \intO{ \Big[ \vr \vu \cdot \partial_t \bfphi + \vr \vu \otimes \vu : \Grad \bfphi 
+ \Ov{p(\vr)} \Div \bfphi - \mathbb{S}  : \Grad \bfphi \Big] }
\end{split}
\end{equation}
for any $\bfphi \in C^1([0,T]; X_n)$. Here $\Ov{p(\vr)} \in L^\infty((0,T) \times \Omega)$ stands for the 
weak limit of the sequence $\{ p(\vre) \}_{\ep > 0}$. As $p = p(\vr)$ is non--linear (convex), removing the bar is 
\emph{equivalent} to showing pointwise convergence of the approximate densities. This might be possible by manipulating 
the renormalized equation in Lemma \ref{EL3}. However, this is quite technical and we content ourselves with \eqref{E20}. 

\subsubsection{Conclusion} 

Finally, employing the weak lower semi--continuity of the potentials $F$ and $F^*$, we may perform the limit in the energy balance 
\eqref{E12} obtaining 
\begin{equation} \label{E21}
\begin{split}
&\left[ \intO{\left[ \frac{1}{2} \vr |\vu - \vu_B|^2 + \Ov{P(\vr)} \right] } \right]_{t = 0}^{ t = \tau} + 
\int_0^\tau \intO{\left[ F(\Ds \vu) + F^*(\mathbb{S}) \right] } \dt \\ 
&+\int_0^\tau \int_{\Gamma_{\rm out}} \Ov{P(\vr)}  \vu_B \cdot \vc{n} \ \D S_x \dt\\
\leq
&- 
\int_0^\tau \intO{ \left[ \vr \vu \otimes \vu + \Ov{p(\vr)} \mathbb{I} \right]  :  \Grad \vu_B } \dt + \int_0^\tau \intO{ {\vr} \vu  \cdot \vu_B \cdot \Grad \vu_B  } 
\dt\\ &+ \int_0^\tau \intO{ \mathbb{S} : \Grad \vu_B } \dt - \int_0^\tau \int_{\Gamma_{\rm in}} P(\vr_B)  \vu_B \cdot \vc{n} \ \D S_x \dt  
\end{split}
\end{equation}

To conclude, we summarize the result obtained in the part. 

\begin{Proposition}[{\bf Approximate solutions, level III}] \label{EP2}

Let $\Omega \subset R^d$, $d=2,3$, be a bounded Lipschitz domain. Suppose that $p = p(\vr)$ and $F = F(\mathbb{D})$ satisfy 
\eqref{S2}, \eqref{S3}--\eqref{S5}. Let the data belong to the class 
\[
\vu_B \in C^{1}_c(R^d; R^d), \ \vr_{0} \in C^{1}(R^d),\ \vr_B \in C^1(\partial \Omega),\ \vr_0, {\vr_B} \geq \underline{\vr} > 0,\ 
\vu_0 = \vv_0 + \vu_B,\ \vv_0 \in X_n.
\]

Then for each fixed $n > 0$, there exists a solution 
$\vr $, $\vu$, 
\[
\vr \in L^\infty((0,T) \times \Omega),\ 0 < \underline{\vr} 
\leq \vr(t,x) \leq \Ov{\vr},\ \vu = \vu_B + \vv ,\ \vv \in C([0,T]; X_n)
\]
of the approximate problem:
\begin{itemize}
\item
\begin{equation} \label{E22}
\begin{split}
\left[ \intO{ \vr \varphi } \right]_{t = 0}^{t = \tau}&= 
\int_0^\tau \intO{ \Big[ \vr \partial_t \varphi + \vr \vu \cdot \Grad \varphi  \Big] }
\dt \\ 
&- \int_0^\tau \int_{\Gamma_{\rm out}} \varphi \vr \vu_B \cdot \vc{n} \ {\rm d} S_x \dt - 
\int_0^\tau \int_{\partial \Gamma_{\rm in}} \varphi \vr_B \vu_B \cdot \vc{n}  \ {\rm d}S_x \dt,\ 
\vr(0, \cdot) = \vr_0 
\end{split}
\end{equation}
for any $\varphi \in C^1([0,T] \times \Ov{\Omega})$;
\item 
\begin{equation} \label{E23}
\begin{split}
\left[ \intO{ \vr \vu \cdot \bfphi } \right]_{t=0}^{t = \tau} &= 
\int_0^\tau \intO{ \Big[ \vr \vu \cdot \partial_t \bfphi + \vr \vu \otimes \vu : \Grad \bfphi 
+ \Ov{p(\vr)} \Div \bfphi - \mathbb{S}  : \Grad \bfphi \Big] }
\end{split}
\end{equation}
for any $\bfphi \in C^1([0,T]; X_n)$;
\item 
\begin{equation} \label{E24}
\begin{split}
&\left[ \intO{\left[ \frac{1}{2} \vr |\vu - \vu_B|^2 + \Ov{P(\vr)} \right] } \right]_{t = 0}^{ t = \tau} + 
\int_0^\tau \intO{\left[ F(\Ds \vu) + F^*(\mathbb{S}) \right] } \dt \\ 
&+\int_0^\tau \int_{\Gamma_{\rm out}} \Ov{P(\vr)}  \vu_B \cdot \vc{n} \ \D S_x \dt\\
\leq
&- 
\int_0^\tau \intO{ \left[ \vr \vu \otimes \vu + \Ov{p(\vr)} \mathbb{I} \right]  :  \Grad \vu_B } \dt + \int_0^\tau \intO{ {\vr} \vu  \cdot \vu_B \cdot \Grad \vu_B  } 
\dt\\ &+ \int_0^\tau \intO{ \mathbb{S} : \Grad \vu_B } \dt - \int_0^\tau \int_{\Gamma_{\rm in}} P(\vr_B)  \vu_B \cdot \vc{n} \ \D S_x \dt.  
\end{split}
\end{equation}
\end{itemize}

The symbol $\Ov{p(\vr)}$ and $\Ov{P(\vr)}$ stand for the weak limits of a bounded sequence 
$\{ p(\vr_\ep) \}_{\ep > 0}$ and $\{ P(\vre) \}_{\ep > 0}$, respectively.

\end{Proposition}

\subsection{Final limit}

Our ultimate goal is to perform the limit $n \to \infty$ in the family of approximate solutions obtained in 
Proposition \ref{EP2}. Our first observation uses the hypotheses \eqref{S2}, namely 
\[
P''(\vr) = \frac{p'(\vr)}{\vr} > 0 \ \vr > 0.
\]
Consequently, $P$ is a strictly convex function and we have 
\[
\Ov{P(\vr)} - P(\vr) \geq 0.
\] 
Next, again by virtue of hypothesis \eqref{S2}, 
\begin{equation} \label{E25a}
\Ov{P(\vr)} - P(\vr) \geq \underline{a} (\Ov{p(\vr)} - p(\vr) ),\ 
\Ov{a} \left( \Ov{p(\vr)} - p(\vr) \right) \geq \Ov{P(\vr)} - P(\vr).
\end{equation}
Thus we may rewrite \eqref{E23} as 
\begin{equation} \label{E25}
\begin{split}
\left[ \intO{ \vr \vu \cdot \bfphi } \right]_{t=0}^{t = \tau} &= 
\int_0^\tau \intO{ \Big[ \vr \vu \cdot \partial_t \bfphi + \vr \vu \otimes \vu : \Grad \bfphi 
+ {p(\vr)} \Div \bfphi - \mathbb{S}  : \Grad \bfphi \Big] }\\ 
&+ \int_0^\tau \intO{ \left[ \Ov{p(\vr)} - p(\vr) \right]  \Div \bfphi } \dt,
\end{split}
\end{equation}
while the energy inequality \eqref{E24} reduces to 
\begin{equation} \label{E26}
\begin{split}
&\left[ \intO{{\left( \frac{1}{2} \frac{|\vm|^2}{\vr} - \vm \cdot \vu_B + \frac{1}{2}\vr |\vu_B|^2  + P(\vr)\right)}} \right]_{t = 0}^{ t = \tau} + 
\int_0^\tau \intO{\left[ F(\Ds \vu) + F^*(\mathbb{S}) \right] } \dt \\ 
&+\int_0^\tau \int_{\Gamma_{\rm out}} {P(\vr)}  \vu_B \cdot \vc{n} \ \D S_x \dt + 
\intO{ \left[ \Ov{P(\vr)} - P(\vr) \right] (\tau, \cdot)} 
\\
\leq
&- 
\int_0^\tau \intO{ \left[ {\frac{\vm \otimes \vm}\vr} + {p(\vr)} \mathbb{I} \right]  :  \Grad \vu_B } \dt + \int_0^\tau \intO{ {\vr} \vu  \cdot \vu_B \cdot \Grad \vu_B  } 
\dt\\ &+ \int_0^\tau \intO{ \mathbb{S} : \Grad \vu_B } \dt - \int_0^\tau \int_{\Gamma_{\rm in}} P(\vr_B)  \vu_B \cdot \vc{n} \ \D S_x \dt\\
&+ \int_0^\tau \intO{ \left( p(\vr) - \Ov{p(\vr)} \right) \Div \vu_B } \dt,  
\end{split}
\end{equation}
{where the kinetic energy 
$\frac{1}{2} \vr |\vu - \vu_B|^2$
in \eqref{E24} is now more conveniently  written as
$\frac{1}{2} \frac{|\vm|^2}{\vr}\ -\vm \cdot \vu_B$ $ + \frac{1}{2}\vr |\vu_B|^2$ and similarly $\vr\vu\otimes\vu$ is written as $\frac{\vm\otimes\vm}\vr$, where $\vm=\vr\vu$.
We set }
\[
\mathfrak{R} := [ \Ov{p(\vr)} - p(\vr) ] \mathbb{I},\ \mathfrak{E} := \Ov{P(\vr)} - P(\vr).
\]
{These terms may be interpreted as the Reynolds and energy defect measure, respectively.} In accordance with \eqref{E25a}, we have 
the defect compatibility condition
\begin{equation} \label{E27}
\mathfrak{E} = \Ov{P(\vr)} - P(\vr) \leq \Ov{a} \left(\Ov{p(\vr)} - p(\vr) \right) = 
\frac{\Ov{a}}{d}{\rm tr}[ \mathfrak{R} ] \leq \frac{\Ov{a}}{\underline{a}} \left( \Ov{P(\vr)} - P(\vr) 
\right) = \frac{\Ov{a}}{\underline{a}} \mathfrak{E}.
\end{equation} 

We are ready to perform the limit $n \to \infty$. Let $[\vr_n, \vm_n]$ be a sequence of solutions obtained in Proposition \ref{EP2},
with the associated viscous stress tensors $\mathbb{S}_n$, the energy defects $\mathfrak{E}_n$, and the Reynolds tenors $\mathfrak{R}_n$. An easy application of Gronwall's lemma shows that the total energy represented by the epxression on the left--hand side of the energy inequality \eqref{E26} remains bounded uniformly for $n \to \infty$. Consequently, extracting suitable subsequences if necessary, we may suppose 
\[
\vr_n \to \vr \ \mbox{in}\ C_{{\rm weak}}([0,T]; L^\gamma(\Omega)),\ 
\vr_n|_{\Gamma_{\rm out}} \to \vr \ \mbox{weakly-(*) in}\ L^\infty(0,T; L^\gamma (\Gamma_{\rm out}; |\vu_b \cdot \vc{n}| \D S_x)), 
\]
\[
\vm_n = \vr_n \vu_n \to \vm \ \mbox{weakly-(*) in}\ L^\infty(0,T; L^{\frac{2 \gamma}{\gamma + 1}}(\Omega; R^d)).
\]

Next, repeating the arguments leading to \eqref{E10}, we get 
\[
\vu_n \to \vu \ \mbox{weakly in}\ L^q(0,T; W^{1,q}(\Omega; R^d)), 
\]
and also 
\[
\mathbb{S}_n \to \mathbb{S} \ \mbox{weakly in}\ L^1((0,T) \times \Omega; R^{d \times d}_{\rm sym}).
\]
Our goal is to show that 
\begin{equation} \label{E29}
\vm = \vr \vu \ \mbox{a.a. in}\ (0,T) \times \Omega. 
\end{equation}
To this end, we report the following result proved in Appendix. 

\begin{Lemma} \label{EWANL1}

Let $Q = (0,T) \times \Omega$, where $\Omega \subset R^d$ is a bounded domain. Suppose that 
\[
r_n \to r \ \mbox{weakly in}\ L^p(Q), \ v_n \to v \ \mbox{weakly in}\ L^q(Q),\ p > 1, q > 1, 
\]
and
\[
r_n v_n \to w \ \mbox{weakly in}\ L^r(Q), \ r >  1.
\]
In addition, let 
\[
\partial_t r_n = \Div \vc{g}_n + h_n\ \mbox{in}\ \mathcal{D}'(Q),\ \| \vc{g}_n \|_{L^s(Q; R^d)} \aleq 1,\ s > 1,\ 
h_n \ \mbox{precompact in}\ W^{-1,z},\ z > 1, 
\]
and
\[
\left\| \Grad v_n \right\|_{\mathcal{M}(Q; R^d)} \aleq 1 \ \mbox{uniformly for}\ n \to \infty.
\]

Then 
\[
w = r v \ \mbox{a.a. in}\ Q.
\]

\end{Lemma}

A direct application of Lemma \ref{EWANL1} to 
\[
r_n = \vr_n,\ v_n = u^i_n,\ i=1,\dots,d, \ \vc{g}_n = - \vr_n \vu_n,\ h_n = 0,\ p = \gamma,\ r = s = 
\frac{2 \gamma}{\gamma + 1}
\]
yields \eqref{E29}.

At this stage, we are able to perform the limit in the equation of continuity \eqref{E22} to obtain \eqref{W1}. We can also approximate 
the initial data 
\[
\vr_{0,n} \to \vr_0 \ \mbox{in}\ L^\gamma(\Omega)
\]
to obtain \eqref{W2} with the desired finite energy initial data. 

The next step is to perform the same limit in the momentum equation \eqref{E25}. To this end, we first observe that 
\[
\mathfrak{R}_n \equiv [\Ov{p(\vr)} - p(\vr_n)]_n \to \mathfrak{R}^1 
\ \mbox{weakly-(*) in}\ L^\infty(0,T; \mathcal{M}(\Ov{\Omega})),\ 
\]
\[
\mathfrak{E}_n \equiv [\Ov{P(\vr)} - P(\vr_n)]_n \to \mathfrak{E}^1 
\ \mbox{weakly-(*) in}\ L^\infty(0,T; \mathcal{M}(\Ov{\Omega})),
\]
where the limit measures retain the compatibility condition \eqref{E27}
\[
0 \leq \mathfrak{E}^1 \leq 
{\Ov{a}}\mathfrak{R}^1 \leq \frac{\Ov{a}}{\underline{a}} \mathfrak{E}^1.
\]

Similarly, we have 
\[
\begin{split}
\vr_n \vu_n \otimes \vu_n + p(\vr_n) \mathbb{I} &= 
1_{\vr_n > 0} \frac{\vm_n \otimes \vm_n}{\vr_n} + p(\vr_n) \mathbb{I} 
\to \Ov{ \left[ 1_{\vr > 0} \frac{\vm \otimes \vm}{\vr} + p(\vr) \mathbb{I} \right]} \\ 
&\mbox{weakly-(*) in}\ L^\infty(0,T; \mathcal{M}(\Ov{\Omega}; R^{d \times d}_{\rm sym})),
\end{split}
\]
and 
\[
\begin{split}
\frac{1}{2} \vr_n |\vu_n|^2 + P(\vr_n) &= 
\frac{1}{2} \frac{|\vm_n|^2}{\vr_n} + P(\vr_n) \to \Ov{ \left[ \frac{1}{2} \frac{|\vm|^2}{\vr} + P(\vr) \right] }
\\ 
&\mbox{weakly-(*) in}\ L^\infty(0,T; \mathcal{M}(\Ov{\Omega})).
\end{split}
\]
We set 
\[
\begin{split}
\mathfrak{R}^2 &= \Ov{ \left[ 1_{\vr > 0} \frac{\vm \otimes \vm}{\vr} + p(\vr) \mathbb{I} \right]} - 
\left( 1_{\vr > 0} \frac{\vm \otimes \vm}{\vr} + p(\vr) \mathbb{I} \right)\\ &= \Ov{ \left[ 1_{\vr > 0} \frac{\vm \otimes \vm}{\vr} + p(\vr) \mathbb{I} \right]} - \left( \vr \vu \otimes \vu + p(\vr) \mathbb{I} \right),
\end{split}
\]
and 
\[
\mathfrak{E}^2 = 
\Ov{ \left[ \frac{1}{2} \frac{|\vm|^2}{\vr} + P(\vr) \right] } - 
\left[ \frac{1}{2} \frac{|\vm|^2}{\vr} + P(\vr) \right] = \Ov{ \left[ \frac{1}{2} \frac{|\vm|^2}{\vr} + P(\vr) \right] }
- \left[ \frac{1}{2} \vr |\vu|^2 + P(\vr) \right] 
\]
noting the relation 
\[
\underline{d} \mathfrak{E}^2 \leq {\rm tr}[ \mathfrak{R}^2] \leq \Ov{d} \mathfrak{E}^2, 
\ \mbox{where}\ 0 < \underline{d} \leq \Ov{d}, \ \underline{d} = \underline{d} (\underline{a}, \Ov{a}, d). 
\]
Finally, we claim that 
\[
\mathfrak{R}^2 \in L^\infty(0,T; \mathcal{M}^+ (\Ov{\Omega}; R^{d \times d}_{\rm sym})).
\]
To see this, it is enough to observe that 
\[
\Ov{ 1_{\vr > 0} \frac{ \vm \otimes \vm }{\vr} } - 1_{\vr > 0} \frac{\vm \otimes \vm}{\vr} \geq 0.
\]
Indeed we compute 
\[
\left[ \Ov{ 1_{\vr > 0} \frac{ \vm \otimes \vm }{\vr} } - 1_{\vr > 0} \frac{\vm \otimes \vm}{\vr} \right]
:(\xi \otimes \xi) = \left[ \Ov{ \frac{ |\vm \cdot \xi|^2 }{\vr} } - \frac{ |\vm \cdot \xi|^2 }{\vr} \right] \geq 0
\ \mbox{for any}\ \xi \in R^d,
\]
{where the most right inequality follows from convexity of the l.s.c. function
$$
[\vr, \vm] \mapsto \left\{\begin{array}{c}
 \frac{|\vm \cdot \xi|^2}{\vr}\;\mbox{if $\vr>0$},
 \\
 0\;\mbox{if $\vr=0$, $\vm=0$},\\
 \infty\;\mbox{otherwise}.
                   \end{array}\right.
$$
}

Having collected all necessary material, we are now ready to send $n \to \infty$ in both the momentum balance \eqref{E25} 
and the energy inequality \eqref{E26}. In particular, we obtain 
\begin{equation} \label{E30}
\begin{split}
\left[ \intO{ \vr \vu \cdot \bfphi } \right]_{t=0}^{t = \tau} &= 
\int_0^\tau \intO{ \Big[ \vr \vu \cdot \partial_t \bfphi + \vr \vu \otimes \vu : \Grad \bfphi 
+ {p(\vr)} \Div \bfphi - \mathbb{S}  : \Grad \bfphi \Big] }\\ 
&+ \int_0^\tau \left( \int_{\Omega} \Grad \bfphi: \D \mathfrak{R}(t) \right) \dt,\ 
\mathfrak{R} = \mathfrak{R}^1 + \mathfrak{R^2},
\end{split}
\end{equation}
for any test function $\bfphi \in C^1([0,T]; X_n)$, $n$ arbitrary. It is a routine matter to choosed the spaces $X_n$ in such a way that validity of \eqref{E30} can be extended to 
$\bfphi \in C^1([0,T]; C^1_c(\Omega))$ by density argument. Finally, for a function 
\begin{equation} \label{E31}
\bfphi \in C^1([0,T] \times \Ov{\Omega}; R^d),\ \bfphi|_{\partial \Omega} = 0,
\end{equation}
we construct a sequence $\bfphi_n \in C^1([0,T]; C^1_c(\Omega))$ such that 
\[
\begin{split}
&\left\| \bfphi_n \right\|_{W^{1,\infty}(0,T) \times \Omega; R^d)} \leq c 
\ \mbox{uniformly for}\ n \to \infty,\\ 
\bfphi_n (t,x) &\to \bfphi (t,x) ,\ \partial_t \bfphi_n (t,x) \to \partial_t \bfphi (t,x),\ 
\Grad \bfphi_n (t,x) \to \Grad \bfphi (t,x) 
\ \mbox{for \emph{any}}\ (t,x) \in (0,T) \times \Omega. 
\end{split}
\]
In such a way we can extend validity of \eqref{E30} to the class of test function \eqref{E31}. 

We have shown the following \emph{existence result}.

\begin{Theorem} [{\bf Global existence of dissipative solutions}] \label{ET1}

Let $\Omega \subset R^d$, $d=2,3$ be a bounded Lipschitz domain. 
Suppose that $p = p(\vr)$ and $F = F(\mathbb{D})$ satisfy 
\eqref{S2}, \eqref{S3}--\eqref{S5}. Let the data belong to the class 
\[
\vu_B \in C^{1}(R^d; R^d), \ \vr_B \in C^1(R^d),\ \vr_B \geq \underline{\vr} > 0, 
\]
\[
\vr_0 \in L^\gamma (\Omega),\ \vr_0 \geq 0,\ \vm_0 \in L^{\frac{2 \gamma}{\gamma + 1}}(\Omega; R^d), 
\intO{ \left[ \frac{1}{2} \frac{|\vm_0|^2}{\vr_0} + P(\vr_0) \right] } < \infty. 
\]

Then the problem \eqref{P1}--\eqref{P7} admits at least one dissipative solution $[\vr, \vu]$ in $(0,T) \times \Omega$ in the sense specified in Definition \ref{WD1}.

\end{Theorem}

\section{Compatibility}
\label{C}

We show that if a dissipative solution enjoys certain regularity, specifically if 
\[
\vu \in C^1([0,T] \times \Ov{\Omega}; R^d),\ \vr \in C^1([0,T] \times \Ov{\Omega}),\ 
\inf_{(0,T) \times \Omega} \vr > 0,
\]
then $[\vr, \vu]$ is a classical solution, meaning $\mathfrak{E} = \mathfrak{R} = 0$. 

To see this, we realize that $(\vu - \vu_B)$ can be used as a test function in the momentum equation 
\eqref{W3}, which, together with the equation of continuity \eqref{W1}, yield the total energy \emph{equality}:
\begin{equation} \label{C1}
\begin{split}
&\left[ \intO{\left[ \frac{1}{2} \vr |\vu - \vu_B|^2 + P(\vr) \right] } \right]_{t = 0}^{ t = \tau} + 
\int_0^\tau \intO{ \mathbb{S}: \Ds \vu } \dt  
+\int_0^\tau \int_{\partial \Omega} P(\vr)  \vu_B \cdot \vc{n} \ \D S_x \dt
\\
&+ \int_{\Ov{\Omega}} 1 \D \ \mathfrak{E} (\tau) \\	
=
&- 
\int_0^\tau \intO{ \left[ \vr \vu \otimes \vu + p(\vr) \mathbb{I} \right]  :  \Grad \vu_B } \dt + \int_0^\tau \intO{ {\vr} \vu  \cdot \vu_B \cdot \Grad \vu_B  } 
\dt + \int_0^\tau \intO{ \mathbb{S} : \Grad \vu_B } \dt \\ &- 
\int_0^\tau \int_{\Ov{\Omega}} \Grad (\vu_B - \vu) : \D \ \mathfrak{R}(t) \dt.  
\end{split}
\end{equation}
Relation \eqref{C1} subtracted from the energy inequality \eqref{W5} give rise to 
\[
\int_{\Ov{\Omega}} 1 \D \ \mathfrak{E} (\tau) \leq \int_0^\tau \int_{\Ov{\Omega}} \Grad \vu : \D \ \mathfrak{R}(t) \dt,
\]
which, together with the compatibility hypothesis \eqref{W6} and Gronwall lemma, yields the desired conclusion 
$\mathfrak{E} = \mathfrak{R} = 0$. 

\begin{Theorem} [{\bf Compatibility weak--strong}] \label{CT1}

Let $\Omega \subset R^d$ be a bounded Lipschitz domain. Suppose that $[\vr, \vu]$ is a dissipative solution in the sense of 
Definition \ref{WD1} belonging to the class 
\[
\vu \in C^1([0,T] \times \Ov{\Omega}; R^d),\ \vr \in C^1([0,T] \times \Ov{\Omega}),\ 
\inf_{(0,T) \times \Omega} \vr > 0.
\]

Then $[\vr, \vu]$ is a classical solution, meaning $\mathfrak{E} = \mathfrak{R} = 0$ and the 
equations are satisfied in the classical sense.

\end{Theorem}

\section{Relative energy}
\label{RR}

The relative energy is a basic tool for showing the weak--strong uniqueness property. Let us introduce
\[
\mathcal{E} \left(\vr, \vu \ \Big|\ \tvr, \tvu \right) = \frac{1}{2} \vr |\vu - \tvu|^2 + P(\vr) - 
P'(\tvr) (\vr - \tvr) - P(\tvr) 
\]
that can be rewritten as   
\[
\begin{split}
\frac{1}{2} \vr |\vu - \tvu|^2 &+ P(\vr) - 
P'(\tvr) (\vr - \tvr) - P(\tvr) \\
&= \frac{1}{2} \vr |\vu - \vu_B - (\tvu - \vu_B)|^2  + P(\vr) - 
P'(\tvr) (\vr - \tvr) - P(\tvr) 
\\&= \left[ \frac{1}{2} \vr |\vu - \vu_B|^2 + P(\vr) \right] - \vr \vu \cdot (\tvu - \vu_B) 
 + \left[ \frac{1}{2} \Big( |\tvu|^2 - |\vu_B|^2 \Big) -  P'(\tvr) \right] \vr + 
p(\tvr)
\end{split}
\]
Our goal is to evaluate the time evolution of 
\[
\intO{ \mathcal{E} \left(\vr, \vu \ \Big|\ \tvr, \tvu \right) }
\]
where $[\vr, \vu]$ is a dissipative solutions and $[\tvr, \tvu]$ are test functions in the class 
\[
\tvu \in C^1([0,T] \times \Ov{\Omega}; R^d),\ 
\tvu|_{\partial \Omega} = \vu_B|_{\partial \Omega},\ 
\tvr \in C^1([0, T] \times \Ov{\Omega}), \ \inf_{(0,T) \times \Omega} \tvr > 0.
\]

\medskip

{\bf Step 1:}

In accordance with the energy inequality \eqref{W5}, we get 
\begin{equation} \label{RR1}
\begin{split}
&\left[ \intO{\left[ \frac{1}{2} \vr |\vu - \vu_B|^2 + P(\vr) \right] } \right]_{t = 0}^{ t = \tau} + 
\int_0^\tau \intO{ \Big[ F(\Ds \vu) + F^* (\mathbb{S}) \Big] } \dt  
\\
& { +\int_0^\tau \int_{\Gamma_{\rm out}} P(\vr)  \vu_B \cdot \vc{n} \ \D S_x \dt +\int_0^\tau \int_{\Gamma_{\rm in}} P(\vr_B)  \vu_B \cdot \vc{n} \ \D S_x \dt}
+ \int_{\Ov{\Omega}} 1 \D \ \mathfrak{E} (\tau) \\	
\leq
&- 
\int_0^\tau \intO{ p(\vr)  \Div \vu_B } \dt + \int_0^\tau \intO{ {\vr} (\vu_B - \vu ) \cdot \Big(\vu \cdot \Grad \vu_B  \Big) } 
\dt + \int_0^\tau \intO{ \mathbb{S} : \Grad \vu_B } \dt \\ &- 
\int_0^\tau \int_{\Ov{\Omega}} \Grad \vu_B : \D \ \mathfrak{R}(t) \dt 
\end{split}
\end{equation}

\medskip

{\bf Step 2:}

Plugging $\bfphi = \tvu - \vu_B$ in the momentum equation \eqref{W3}, we get
\begin{equation} \label{RR2}
\begin{split}
&\left[ \intO{ \vr \vu \cdot (\tvu - \vu_B) } \right]_{t=0}^{t = \tau} \\ &= 
\int_0^\tau \intO{ \Big[ \vr \vu \cdot \partial_t \tvu  + \vr \vu \otimes \vu : \Grad (\tvu - \vu_B) 
+ p(\vr) \Div (\tvu - \vu_B) - \mathbb{S} : \Grad (\tvu - \vu_B) \Big] } \dt \\
&+ \int_0^\tau \int_{\Ov{\Omega}} \Grad (\tvu - \vu_B) : \D \ \mathfrak{R}(t) \ \dt
\end{split}
\end{equation}

\medskip

{\bf Step 3:}

Finally, we consider $\varphi = \left[ \frac{1}{2} \Big( |\tvu|^2 - |\vu_B|^2 \Big) -  P'(\tvr) \right]$ in 
the equation of continuity \eqref{W1} obtaining: 
\begin{equation} \label{RR3}
\begin{split}
&\left[ \intO{ \vr \left[ \frac{1}{2} \Big( |\tvu|^2 - |\vu_B|^2 \Big) -  P'(\tvr) \right] } \right]_{t = 0}^{t = \tau}  \\&{ -
\int_0^\tau \int_{\Gamma_{\rm out}}P'(\tilde\vr)\vr \vu_B \cdot \vc{n} \ \D \ S_x 
-
\int_0^\tau \int_{\Gamma_{\rm in}}P'(\tilde\vr)\vr_B \vu_B \cdot \vc{n} \ \D \ S_x 
}
\\ &= 
\int_0^\tau \intO{ \Big[ \vr \partial_t\left( \frac{1}{2} |\tvu|^2 -  P'(\tvr) \right)  + 
\vr \vu \cdot \Grad \left( \frac{1}{2} \Big( |\tvu|^2 - |\vu_B|^2 \Big) -  P'(\tvr) \right) \Big] } \dt 
\end{split}
\end{equation}

Summing up \eqref{RR1}--\eqref{RR3} we get 
\begin{equation} \label{RR4}
\begin{split}
&\left[ \intO{\mathcal{E}\left( \vr, \vu \ \Big|\ \tvr, \tvu \right) } \right]_{t = 0}^{ t = \tau} + 
\int_0^\tau \intO{ \Big[ F(\Ds \vu) + F^* (\mathbb{S}) \Big] } \dt - \int_0^\tau \intO{ \mathbb{S} : \Grad \tvu } \dt \\  
&{+\int_0^\tau \int_{\Gamma_{\rm out}} \left[ P(\vr) - P'(\tvr) \vr \right]  \vu_B \cdot \vc{n} \ \D S_x \dt
+\int_0^\tau \int_{\Gamma_{\rm in}} \left[ P(\vr_B) - P'(\tvr) \vr_B \right]  \vu_B \cdot \vc{n} \ \D S_x \dt}
\\
&+ \int_{\Ov{\Omega}} 1 \D \ \mathfrak{E} (\tau) \\	
\leq 
&- \int_0^\tau \intO{ \Big[ \vr \vu \cdot \partial_t \tvu  + \vr \vu+
\vu_B \cdot \vc{n} \ \D S_x \dt \otimes \vu : \Grad \tvu 
+ p(\vr) \Div \tvu  \Big] } \dt \\
&- \int_0^\tau \int_{\Ov{\Omega}} \Grad \tvu : \D \ \mathfrak{R}(t) \ \dt  + \int_0^\tau \intO{ \partial_t p(\tvr) } \dt\\
&+\int_0^\tau \intO{ \Big[ \vr \partial_t\left( \frac{1}{2} |\tvu|^2 -  P'(\tvr) \right)  + 
\vr \vu \cdot \Grad \left( \frac{1}{2} |\tvu|^2  -  P'(\tvr) \right) \Big] } \dt 
\end{split}
\end{equation}

Finally, regrouping several terms we conclude
\begin{equation} \label{RR5}
\begin{split}
&\left[ \intO{\mathcal{E}\left( \vr, \vu \ \Big|\ \tvr, \tvu \right) } \right]_{t = 0}^{ t = \tau} + 
\int_0^\tau \intO{ \Big[ F(\Ds \vu) + F^* (\mathbb{S}) \Big] } \dt - \int_0^\tau \intO{ \mathbb{S} : \Grad \tvu } \dt \\ 
&{ +\int_0^\tau \int_{\Gamma_{\rm out}} \left[ P(\vr) - P'(\tvr) (\vr - \tvr) - P(\tvr)  \right]  \vu_B \cdot \vc{n} \ \D S_x \dt}\\
&{ +\int_0^\tau \int_{\Gamma_{\rm in}} \left[ P(\vr_B) - P'(\tvr) (\vr_B - \tvr) - P(\tvr)  \right]  \vu_B \cdot \vc{n} \ \D S_x \dt}
+ \int_{\Ov{\Omega}} 1 \D \ \mathfrak{E} (\tau) \\	
\leq &- \int_0^\tau \intO{ \vr (\tvu - \vu) \cdot (\tvu - \vu) \cdot \Grad \tvu } \dt\\
&- 
\int_0^\tau \intO{ \Big[ p(\vr) - p'(\tvr) (\vr - \tvr) - p(\tvr) \Big] \Div \tvu } \dt   
\\ 
&+ \int_0^\tau \intO{ \frac{\vr}{\tvr} (\tvu - \vu) \cdot \Big[ \partial_t (\tvr \tvu)  +  \Div (\tvr \tvu \otimes \tvu) 
  + \Grad p(\tvr) \Big] } \dt\\ &+ \int_0^\tau \intO{ \left( \frac{\vr}{\tvr} (\vu - \tvu) \cdot \tvu +
	p'(\tvr)\left( 1 - \frac{\vr}{\tvr} \right) \right) \Big[ \partial_t \tvr   +  \Div (\tvr \tvu) 
  \Big] } \dt
	\\
&- \int_0^\tau \int_{\Ov{\Omega}} \Grad \tvu : \D \ \mathfrak{R}(t) \ \dt.  
\end{split}
\end{equation}

We have shown the following result. 

\begin{Proposition}[{\bf Relative energy inequality}] \label{RRP1} 

Let $[\vr, \vu]$ be a dissipative solution in the sense of Definition \ref{WD1}. Suppose that 
\begin{equation} \label{RR6}
\tvu \in C^1([0,T] \times \Ov{\Omega}; R^d),\ 
\tvu|_{\partial \Omega} = \vu_B|_{\partial \Omega},\ 
\tvr \in C^1([0, T] \times \Ov{\Omega}), \ \inf_{(0,T) \times \Omega} \tvr > 0.
\end{equation}

Then the relative energy inequality \eqref{RR5} holds for a.a. $0 \leq \tau \leq T$.

\end{Proposition}

\section{Weak--strong uniqueness}
\label{WU}

Our goal is to show that a dissipative solutions coincides with the strong solution emanating from the same initial data { and boundary conditions}.
Assuming the strong solution $[\tvr, \tvu]$ belongs to the class \eqref{RR6}, the obvious idea is to use the relative energy 
inequality \eqref{RR5}. Assuming regularity of the viscous stress $\tilde{\mathbb{S}}$ related to the strong solution, we may rewrite 
\eqref{RR5} as  
\begin{equation} \label{WU1}
\begin{split}
&\left[ \intO{\mathcal{E}\left( \vr, \vu \ \Big|\ \tvr, \tvu \right) } \right]_{t = 0}^{ t = \tau} \\&+ 
\int_0^\tau \intO{ \Big[ F(\Ds \vu) + F^* (\mathbb{S}) \Big] } \dt 
+ \int_0^\tau \intO{ \widetilde{\mathbb{S}}: ( \Ds \tvu - \Ds \vu) } \dt 
- \int_0^\tau \intO{ \mathbb{S} : \Ds \tvu } \dt \\  
&+\int_0^\tau \int_{\Gamma_{\rm out}} \left[ P(\vr) - P'(\tvr) (\vr - \tvr) - P(\tvr)  \right]  \vu_B \cdot \vc{n} \ \D S_x \dt
+ \int_{\Ov{\Omega}} 1 \D \ \mathfrak{E} (\tau) \\	
\leq &- \int_0^\tau \intO{ \vr (\tvu - \vu) \cdot (\tvu - \vu) \cdot \Grad \tvu } \dt\\
&- 
\int_0^\tau \intO{ \Big[ p(\vr) - p'(\tvr) (\vr - \tvr) - p(\tvr) \Big] \Div \tvu } \dt   
\\ 
&+ \int_0^\tau \intO{ \frac{\vr}{\tvr} (\tvu - \vu) \cdot \Big[ \partial_t (\tvr \tvu)  +  \Div (\tvr \tvu \otimes \tvu) 
  + \Grad p(\tvr) - \Div \widetilde{\mathbb{S}} \Big] } \dt\\ &+ \int_0^\tau \intO{ \left( \frac{\vr}{\tvr} (\vu - \tvu) \cdot \tvu +
	p'(\tvr)\left( 1 - \frac{\vr}{\tvr} \right) \right) \Big[ \partial_t \tvr   +  \Div (\tvr \tvu) 
  \Big] } \dt
	\\
&+ \int_0^\tau \intO{ \left( \frac{\vr}{\tvr} - 1 \right) (\tvu - \vu) \cdot \Div \widetilde{\mathbb{S}} } \dt
\\	
&- \int_0^\tau \int_{\Ov{\Omega}} \Grad \tvu : \D \ \mathfrak{R}(t) \ \dt, 
\end{split}
\end{equation}
where we have used the identity 
\[
\int_0^\tau \intO{  (\tvu - \vu) \cdot \Div \widetilde{\mathbb{S}} } \dt = 
- \int_0^\tau \intO{  \Ds (\tvu - \vu) : \widetilde{\mathbb{S}} } \dt.
\]

As $[\tvr, \tvu]$ is a strong solution of the problem with the same initial--boundary data, relation \eqref{WU1} reduces
to 
\begin{equation} \label{WU2}
\begin{split}
&\intO{\mathcal{E}\left( \vr, \vu \ \Big|\ \tvr, \tvu \right) (\tau, \cdot) } \\&+ 
\int_0^\tau \intO{ \Big[ F(\Ds \vu) + F^* (\mathbb{S}) \Big] } \dt 
+ \int_0^\tau \intO{ \widetilde{\mathbb{S}}: ( \Ds \tvu - \Ds \vu) } \dt 
- \int_0^\tau \intO{ \mathbb{S} : \Ds \tvu } \dt \\  
&+\int_0^\tau \int_{\Gamma_{out}} \left[ P(\vr) - P'(\tvr) (\vr - \tvr) - P(\tvr)  \right]  \vu_B \cdot \vc{n} \ \D S_x \dt
+ \int_{\Ov{\Omega}} 1 \D \ \mathfrak{E} (\tau) \\	
\leq &- \int_0^\tau \intO{ \vr (\tvu - \vu) \cdot (\tvu - \vu) \cdot \Grad \tvu } \dt\\
&- 
\int_0^\tau \intO{ \Big[ p(\vr) - p'(\tvr) (\vr - \tvr) - p(\tvr) \Big] \Div \tvu } \dt   
\\ 
&+ \int_0^\tau \intO{ \left( \frac{\vr}{\tvr} - 1 \right) (\tvu - \vu) \cdot \Div \widetilde{\mathbb{S}} } \dt
\\	
&- \int_0^\tau \int_{\Ov{\Omega}} \Grad \tvu : \D \ \mathfrak{R}(t) \ \dt. 
\end{split}
\end{equation}

Moreover, it follows from the structural hypothesis \eqref{S2} and the compatibility condition \eqref{W6} that 
\[
\begin{split}
&- \int_0^\tau \intO{ \vr (\tvu - \vu) \cdot (\tvu - \vu) \cdot \Grad \tvu } \dt\\
&- 
\int_0^\tau \intO{ \Big[ p(\vr) - p'(\tvr) (\vr - \tvr) - p(\tvr) \Big] \Div \tvu } \dt   
\\ 
&- \int_0^\tau \int_{\Ov{\Omega}} \Grad \tvu : \D \ \mathfrak{R}(t) \ \dt \\
&\leq c\left(\| \Grad \tvu \|_{L^\infty} \right) \left[  
\int_0^\tau \intO{ \mathcal{E}\left(\vr, \vu \ \Big|\ \tvr, \tvu \right) } \dt + 
\int_0^\tau \left(\int_{\Ov{\Omega}} \D \ \mathfrak{E} (t) \right) \dt     \right].
\end{split}
\]
Consequently, \eqref{WU2} gives rise to 
\begin{equation} \label{WU2a}
\begin{split}
&\intO{\mathcal{E}\left( \vr, \vu \ \Big|\ \tvr, \tvu \right) (\tau, \cdot) } \\&+ 
\int_0^\tau \intO{ \Big[ F(\Ds \vu) + F^* (\mathbb{S}) \Big] } \dt 
+ \int_0^\tau \intO{ \widetilde{\mathbb{S}}: ( \Ds \tvu - \Ds \vu) } \dt 
- \int_0^\tau \intO{ \mathbb{S} : \Ds \tvu } \dt \\  
&+\int_0^\tau \int_{\Gamma_{out}} \left[ P(\vr) - P'(\tvr) (\vr - \tvr) - P(\tvr)  \right]  \vu_B \cdot \vc{n} \ \D S_x \dt
+ \int_{\Ov{\Omega}} 1 \D \ \mathfrak{E} (\tau) \\	
&\leq c\left(\| \Grad \tvu \|_{L^\infty} \right) \left[  
\int_0^\tau \intO{ \mathcal{E}\left(\vr, \vu \ \Big|\ \tvr, \tvu \right) } \dt + 
\int_0^\tau \left(\int_{\Ov{\Omega}} \D \ \mathfrak{E} (t) \right) \dt     \right]
\\ 
&+ \int_0^\tau \intO{ \left( \frac{\vr}{\tvr} - 1 \right) (\tvu - \vu) \cdot \Div \widetilde{\mathbb{S}} } \dt.
\end{split}
\end{equation}

To conclude, we regroup the dissipative integrals on the left hand side as 
\[
\begin{split}
\int_0^\tau \intO{ \Big[ F(\Ds \vu) + F^* (\mathbb{S}) \Big] } \dt 
&+ \int_0^\tau \intO{ \widetilde{\mathbb{S}}: ( \Ds \tvu - \Ds \vu) } \dt 
- \int_0^\tau \intO{ \mathbb{S} : \Ds \tvu } \dt\\
&= \int_0^\tau \intO{ \left[ F(\Ds \vu) - \widetilde{\mathbb{S}}: ( \Ds \vu - \Ds \tvu) - F(\Ds \tvu) \right] } \dt\\
&+ \int_0^\tau \intO{ \left[ F(\Ds \tvu) + F^*(\mathbb{S}) - \Ds \tvu : \mathbb{S} \right] } \dt \\
&\geq \int_0^\tau \intO{ \left[ F(\Ds \vu) - \widetilde{\mathbb{S}}: ( \Ds \vu - \Ds \tvu) - F(\Ds \tvu) \right] } \dt,
\end{split}
\]
where we have used Fenchel--Young inequality. By virtue of the coercivity hypothesis \eqref{S5}, we have 
\begin{equation} \label{WU4}
\begin{split}
\int_0^\tau \intO{ \left[ F(\Ds \vu) - \widetilde{\mathbb{S}}: ( \Ds \vu - \Ds \tvu) - F(\Ds \tvu) \right] } \dt\\
\geq \int_0^\tau \intO{ A_R \left( \left| \Ds (\vu - \tvu) - \frac{1}{d} \Div (\vu - \tvu) \mathbb{I} \right| \right)} \dt,
\end{split}
\end{equation}
where $R = R \left(\| \tvu \|_{L^\infty}, \left\| \widetilde{\mathbb{S}} \right\|_{L^\infty} \right)$.

Finally, we use the following two results concerning Korn and Poincar\' e inequalities:  

\begin{Lemma}[{\bf Talenti \cite[Lemma 3]{Tale}}] \label{WUL1}

Let $A$ be a Young function, and let $\Omega \subset R^d$ be a bounded domain. 

Then
\[
\intO{ A \left( \frac{\chi_d}{d |\Omega|^{\frac{1}{d}}} \left| w \right| \right) }
\leq \frac{1}{d} \intO{ A \left( \left| \Grad w \right| \right) }
\] 
for any $w \in W^{1,1}_0(\Omega)$, where $\chi_d$ is a positive constant.

\end{Lemma}

\begin{Lemma}[{\bf Breit, Cianchi, and Diening \cite[Theorem 3.1]{BrCiDi}}] \label{WUL2}

Let $\Omega \subset R^d$, $d \geq 2$ be a bounded domain. Let $A$ be a Young function satisfying the $\Delta^2_2$--condition 
\eqref{S4}.

Then there exists a constant $c > 0$ such that 
\[
\intO{ A ( |\Grad \vc{w}| ) } \leq \intO{ A \left( c \left| \Ds \vc{w} - \frac{1}{d} \Div \vc{w} \mathbb{I} \right| \right) }
\]
for any $\vc{w} \in W^{1,1}_0(\Omega; R^d)$.

\end{Lemma}

Combining Lemmas \ref{WUL1}, \ref{WUL2} with \eqref{WU4} we may infer that  

\begin{equation} \label{WU5}
\begin{split}
&\intO{\mathcal{E}\left( \vr, \vu \ \Big|\ \tvr, \tvu \right) (\tau, \cdot) } + 
\int_0^\tau \intO{ \tilde{A}_R (|\tvu - \vu|) } \dt
\\
&+\int_0^\tau \int_{\Gamma_{out}} \left[ P(\vr) - P'(\tvr) (\vr - \tvr) - P(\tvr)  \right]  \vu_B \cdot \vc{n} \ \D S_x \dt
+ \int_{\Ov{\Omega}} 1 \D \ \mathfrak{E} (\tau) \\	
&\leq c\left(\| \Grad \tvu \|_{L^\infty} \right) \left[  
\int_0^\tau \intO{ \mathcal{E}\left(\vr, \vu \ \Big|\ \tvr, \tvu \right) } \dt + 
\int_0^\tau \left(\int_{\Ov{\Omega}} \D \ \mathfrak{E} (t) \right) \dt     \right]
\\ 
&+ \int_0^\tau \intO{ \left( \frac{\vr}{\tvr} - 1 \right) (\tvu - \vu) \cdot \Div \widetilde{\mathbb{S}} } \dt,
\end{split}
\end{equation}
where $\tilde{A}_R$ is a Young function obtain by a simple rescaling of $A_R$ in \eqref{WU4}. 

Finally, we observe that 
\[
\intO{ \Big| \frac{\vr}{\tvr} - 1 \Big| \Big| \tvu - \vu \Big|
 } = \int_{\{ \vr < \delta \} } \Big| \frac{\vr}{\tvr} - 1 \Big| \Big| \tvu - \vu \Big| \ \dx + 
\int_{\{ \vr \geq \delta \} } \Big| \frac{\vr}{\tvr} - 1 \Big| \Big| \tvu - \vu \Big| \ \dx
\]
for any $\delta > 0$, where, on one hand,  
\[
\int_{\{ \vr  \geq \delta \} } \Big| \frac{\vr}{\tvr} - 1 \Big| \Big| \tvu - \vu \Big|
\dx \leq c(\delta) \intO{ \E \left( \vr, \vm \Big| \tvr , \tvu \right) }.
\]
On the other hand, 
\[
\left| \int_{\{ \varrho < \delta \} } \Big| \frac{\vr}{\tvr} - 1 \Big| \Big| \tvu - \vu \Big|
\dx \right| \leq \frac{1}{2} \intO{ \tilde{A}_R \left( |\tvu - \vu| \right) } + c(\delta, \tilde{A}_R) \intO{ \E \left( \vr, \vm \Big| \tvr , \tvu \right) }. 
\]
Thus applying the standard Gronwall argument to \eqref{WU5} we obtain the desired conclusion: 
\[
\vr = \tvr, \ \vu = \tvu, \ \mathfrak{R} = \mathfrak{E} = 0.
\]

We have proved the following result: 

\begin{Theorem} [{\bf Weak--strong uniqueness}] \label{WUT1}

Let $\Omega \subset R^d$ be a bounded Lipschitz domain. 
Suppose that $p$ and $\mathbb{S}$ satisfy the structural hypotheses \eqref{S2}, \eqref{S3}, and \eqref{S5}.
Let $[\vr, \vu]$ be a dissipative solution 
in the sense of Definition \ref{WD1}, and let $[\tvr, \tvu]$ be a strong solution of the same problem belonging to the 
class 
\[
\tvu \in C^1([0,T] \times \Ov{\Omega}; R^d),\ 
\tvr \in C^1([0, T] \times \Ov{\Omega}), \ \inf_{(0,T) \times \Omega} \tvr > 0, 
\]
with the stress tensor $\tilde{\mathbb{S}} \in C([0,T] \times \Ov{\Omega}; R^{d \times d}_{\rm sym})$,\ 
$\Div \tilde{\mathbb{S}} \in C([0,T] \times \Ov{\Omega}; R^{d})$.

Then 
\[
\vr = \tvr, \ \vu = \tvu \ \mbox{in}\ (0,T) \times \Omega,\ \mathfrak{E} = \mathfrak{R} = 0.
\]

\end{Theorem}

\section{Concluding remarks}
\label{D}

A short inspection of the proofs shows that both compatibility and weak--strong uniqueness property hold if the defect compatibility condition 
\eqref{W6} is weakened to 
\begin{equation} \label{D1}
{\rm tr}[\mathfrak{R}] \leq \Ov{d} \mathfrak{E},\ 
\ \mbox{for a certain constant}\ \Ov{d} > 0.
\end{equation}
The weak formulation can be then written in a concise way as: 
\begin{equation} \label{D2}
\begin{split}
\left[ \intO{ \vr \varphi } \right]_{t = 0}^{t = \tau} &+ 
\int_0^\tau \int_{\Gamma_{\rm out}} \varphi \vr \vu_B \cdot \vc{n} \ \D \ S_x 
+ 
\int_0^\tau \int_{\Gamma_{\rm in}} \varphi \vr_B \vu_B \cdot \vc{n} \ \D \ S_x\\ &= 
\int_0^\tau \intO{ \Big[ \vr \partial_t \varphi + \vr \vu \cdot \Grad \varphi \Big] } \dt,\ \vr(0, \cdot) = \vr_0, 
\end{split}
\end{equation}
for any $\varphi \in C^1([0,T] \times \Ov{\Omega})$;
\begin{equation} \label{D3}
\begin{split}
\left[ \intO{ \vr \vu \cdot \bfphi } \right]_{t=0}^{t = \tau} &= 
\int_0^\tau \intO{ \Big[ \vr \vu \cdot \partial_t \bfphi + \vr \vu \otimes \vu : \Grad \bfphi 
+ p(\vr) \Div \bfphi - \mathbb{S} : \Grad \bfphi \Big] }\\
&+ \int_0^\tau \int_{{\Omega}} \Grad \bfphi : \D \ \mathfrak{R}(t) \ \dt,\ \vr \vu(0, \cdot) = \vm_0,
\end{split}
\end{equation}
for any function $\bfphi \in C^1([0,T] \times \Ov{\Omega}; R^d)$, $\bfphi|_{\partial \Omega} = 0$;
\begin{equation} \label{D4}
\begin{split}
&\left[ \intO{\left[ \frac{1}{2} \vr |\vu - \vu_B|^2 + P(\vr) \right] } \right]_{t = 0}^{ t = \tau} + 
\int_0^\tau \intO{ \Big[ F(\Ds \vu) + F^* (\mathbb{S}) \Big] } \dt\\  
&+\int_0^\tau \int_{\Gamma_{\rm out}} P(\vr)  \vu_B \cdot \vc{n} \ \D S_x \dt +
\int_0^\tau \int_{\Gamma_{\rm in}} P(\vr_B)  \vu_B \cdot \vc{n} \ \D S_x \dt
+ \frac{1}{\Ov{d}} \int_{\Ov{\Omega}}  \D \ {\rm tr}[\mathfrak{R}] (\tau) \\	
\leq
&- 
\int_0^\tau \intO{ \left[ \vr \vu \otimes \vu + p(\vr) \mathbb{I} \right]  :  \Grad \vu_B } \dt + \int_0^\tau \intO{ {\vr} \vu  \cdot \vu_B \cdot \Grad \vu_B  } 
\dt + \int_0^\tau \intO{ \mathbb{S} : \Grad \vu_B } \dt \\ &- 
\int_0^\tau \int_{\Ov{\Omega}} \Grad \vu_B : \D \ \mathfrak{R}(t) \dt, \ \mbox{for some}\ \Ov{d} > 0.
\end{split}
\end{equation}
Note that the energy defect measure $\mathfrak{E}$ is entirely eliminated and the only ``free'' quantity in \eqref{D2}--\eqref{D4}
is the Reynolds stress $\mathfrak{R} \in L^\infty(0,T; \mathcal{M}^+(\Ov{\Omega}; R^d))$. This new definition is in fact 
\emph{equivalent} to Definition \ref{WD1} as one can always define the ``energy defect'' as 
\[
\mathfrak{E} = \frac{1}{\Ov{d}}{\rm trace}[\mathfrak{R}].
\]

Convexity of the pressure $p$ was necessary for the Reynolds stress $\mathfrak{R}$ to be a positively semi-definite tensor.
From the point of view of physics, hypothesis \eqref{S2} may seem too restrictive.
In particular, the physically relevant case of 
the isothermal pressure $p(\vr) = \theta \vr$, $\theta > 0$ is not included. A brief inspection on the proofs reveals that all principal results remain valid for any EOS of the form 
\[
p(\vr) + a \vr, \ a \geq 0
\]
as long as $p$ satisfies \eqref{S2}.

\section{Appendix}

Our goal is to show the following result.

\begin{Lemma} \label{WANL1}

Let $Q = (0,T) \times \Omega$, where $\Omega \subset R^d$ is a bounded domain. Suppose that 
\[
r_n \to r \ \mbox{weakly in}\ L^p(Q), \ v_n \to v \ \mbox{weakly in}\ L^q(Q),\ p > 1, q > 1, 
\]
and
\[
r_n v_n \to w \ \mbox{weakly in}\ L^r(Q), \ r >  1.
\]
In addition, let 
\[
\partial_t r_n = \Div \vc{g}_n + h_n\ \mbox{in}\ \mathcal{D}'(Q),\ \| \vc{g}_n \|_{L^s(Q; R^d)} \aleq 1,\ s > 1,\ 
h_n \ \mbox{precompact in}\ W^{-1,z},\ z > 1, 
\]
and
\[
\left\| \Grad v_n \right\|_{\mathcal{M}(Q; R^d)} \aleq 1 \ \mbox{uniformly for}\ n \to \infty.
\]

Then 
\[
w = r v \ \mbox{a.a. in}\ Q.
\]

\end{Lemma}

\begin{proof}

First, we introduce a cut--off function 
\[
T_k (v) = k \mathcal{T} \left( \frac{v}{k} \right),\ \mathcal{T} \in C^\infty \cap C_B (R),\ 
\mathcal{T}(Z) = \mathcal{T}(-Z),\ T(Z) = Z \ \mbox{if}\ |Z| \leq 1, \ 0 \leq T'(Z) \leq 1. 
\]
Next, write  
\[
v_n = T_k(v_n) + \Big( v_n - T_k(v_n) \Big), 
\]
and
\[
r_n v_n = r_n T_k(v_n) + r_n \Big( v_n - T_k(v_n) \Big). 
\]
Passing to a subsequence (not relabeled) we may assume 
\[
T_k(v_n) \to \Ov{T_k(v)} \ \mbox{weakly-(*) in} \ L^\infty(Q),\ 
r_n T_k(v_n) \to w_k \ \mbox{weakly in}\ L^r(Q) \ \mbox{as}\ n \to \infty. 
\]

We claim that is is enough to show  
\[
w_k = r \Ov{T_k(v)} \ \mbox{a.a. in}\ Q \ \mbox{for any}\ k \to \infty.
\]
Indeed we have 
\[
\int_Q |v_n - T_k(v_n) | \dxdt \leq \int_{|v_n \geq k|} |v_n| \leq 
|\{ v_n \geq k \}|^{\frac{1}{q'}} \| v_n \|_{L^q(Q)} \to 0 
\ \mbox{as}\ k \to \infty \ \mbox{uniformly in}\ n,
\]
and 
\[
\left\| v - \Ov{T_k(v)} \right\|_{L^1(Q)} \leq \liminf_{n \to \infty} 
\left\| v_n - {T_k(v_n)} \right\|_{L^1(Q)} \to 0 \ \mbox{as}\ k \to \infty.
\]
Similarly 
\[
\begin{split}
\int_Q &\left| r_n (v_n - T_k(v_n) ) \right| \ \dxdt \leq 
\int_{|v_n \geq k|} |r_n v_n| \\ &\leq |\{ v_n \geq k \}|^{\frac{1}{r'}} \| r_n v_n \|_{L^r(Q)} \to 0 
\ \mbox{as}\ k \to \infty \ \mbox{uniformly in}\ n.
\end{split}
\]

Since
\[
\| \Grad v_n \|_{{\mathcal{M}(Q; R^d)}} \aleq 1 \ \Rightarrow \ 
\| \Grad T_k (v_n) \|_{{\mathcal{M}(Q; R^d)}} \aleq 1 \ \mbox{uniformly for}\ n \to \infty,
\]
it is enough to show the conclusion under the assumption 
\[
v_n \to v \ \mbox{weakly-(*) in}\ L^\infty(Q).
\]
To this end, we apply Div-Curl lemma to the vector fields
\[
\begin{split}
\vc{U}_n &= [r_n, - \vc{g}_n] : Q \to R^{d + 1},\ 
{\rm DIV}_{t,x} \vc{U}_n = \partial_t r_n + \Div \vc{g}_n = h_n,\\
\vc{U}_n &\to \vc{U} = [r, \vc{g}] \ \mbox{weakly in}\ L^{{\rm min}\{s,p\}} (Q),
\end{split}
\]
and
\[
\vc{V}_n = [v_n, 0 ]: Q \to R^{d+1},\ 
{\rm CURL}_{t,x} \vc{V}_n \approx \Grad v_n \ \mbox{bounded in}\ \mathcal{M}(Q; R^{d \times d}).
\] 
Applying Div--Curl Lemma we obtain the desired conclusion.

\end{proof}

\def\cprime{$'$} \def\ocirc#1{\ifmmode\setbox0=\hbox{$#1$}\dimen0=\ht0
  \advance\dimen0 by1pt\rlap{\hbox to\wd0{\hss\raise\dimen0
  \hbox{\hskip.2em$\scriptscriptstyle\circ$}\hss}}#1\else {\accent"17 #1}\fi}


\begin{thebibliography}{10}

\bibitem{AbbFei2}
A.~Abbatiello and E.~Feireisl.
\newblock On a class of generalized solution to equations describing
  incompressible viscous fluids.
\newblock {\em Archive Preprint Series}, {\bf arxiv preprint No. 1905.12732},
  2019.
\newblock To appear in Annal. Mat. Pura Appl.

\bibitem{BrCiDi}
D.~Breit, A.~Cianchi, and L.~Diening.
\newblock Trace-free {K}orn inequalities in {O}rlicz spaces.
\newblock {\em SIAM J. Math. Anal.}, {\bf 49}(4):2496--2526, 2017.

\bibitem{BreFeiHof19}
D.~Breit, E.~Feireisl, and M.~Hofmanov{\' a}.
\newblock Solution semiflow to the isentropic {E}uler system.
\newblock {\em Arxive Preprint Series}, {\bf arXiv 1901.04798}, 2019.
\newblock To appear in Arch. Rational Mech. Anal.

\bibitem{ChJiNo}
T.~Chang, B.~J. Jin, and A.~Novotn\'{y}.
\newblock Compressible {N}avier-{S}tokes system with general inflow-outflow
  boundary data.
\newblock {\em SIAM J. Math. Anal.}, {\bf 51}(2):1238--1278, 2019.

\bibitem{ChToZi}
G.-Q. Chen, M.~Torres, and W.~P. Ziemer.
\newblock Gauss-{G}reen theorem for weakly differentiable vector fields, sets
  of finite perimeter, and balance laws.
\newblock {\em Comm. Pure Appl. Math.}, 62(2):242--304, 2009.


\bibitem{CrDoSp}
G.~Crippa, C.~Donadello, and L.~V. Spinolo.
\newblock A note on the initial--boundary value problem for continuity
  equations with rough coefficients.
\newblock {\em HYP 2012 conference proceedings, AIMS Series in Appl. Math.},
  {\bf 8}:957--966, 2014.

\bibitem{DiEbRu}
L.~Diening, C.~Ebmeyer, and M.~R\ocirc{u}\v{z}i\v{c}ka.
\newblock Optimal convergence for the implicit space-time discretization of
  parabolic systems with {$p$}-structure.
\newblock {\em SIAM J. Numer. Anal.}, {\bf 45}(2):457--472, 2007.

\bibitem{EF70}
E.~Feireisl.
\newblock {\em Dynamics of viscous compressible fluids}.
\newblock Oxford University Press, Oxford, 2004.

\bibitem{FeLiMa}
E.~Feireisl, X.~Liao, and J.~M{\' a}lek.
\newblock Global weak solutions to a class of non-{N}ewtonian compressible
  fluids.
\newblock {\em Math. Meth. Appl. Sci}, {\bf 38}:3482--3494, 2015.

\bibitem{Girinon}
V. Girinon
\newblock{Navier-Stokes equations with nonhomogenous boundary conditions in a bounded three-dimensional domain.}
\newblock{\em J. Math. Fluid Mech.} {\bf 13}, 309--339, 2011

{
\bibitem{KwNo}
Y. Kwon, A. Novotny.
\newblock{Dissipative solutions to compressible Navier-Stokes equations with general inflow-outflow data: existence, stability and weak-strong uniqueness}.
\newblock{\bf{ arXiv:1905.02667}}, 2019
}

\bibitem{LI4}
P.-L. Lions.
\newblock {\em Mathematical topics in fluid dynamics, Vol.2, Compressible
  models}.
\newblock Oxford Science Publication, Oxford, 1998.

\bibitem{MAM1}
A.~E. Mamontov.
\newblock Global solvability of the multidimensional {N}avier- {S}tokes
  equations of a compressible fluid with nonlinear viscosity, {I}.
\newblock {\em Siberian Math. J.}, {\bf 40}(2):351--362, 1999.

\bibitem{MAM2}
A.~E. Mamontov.
\newblock Global solvability of the multidimensional {N}avier- {S}tokes
  equations of a compressible fluid with nonlinear viscosity, {I}{I}.
\newblock {\em Siberian Math. J.}, {\bf 40}(3):541--555, 1999.

\bibitem{PloWei}
P.~I. Plotnikov and W.~Weigant.
\newblock Isothermal {N}avier-{S}tokes equations and {R}adon transform.
\newblock {\em SIAM J. Math. Anal.}, {\bf 47}(1):626--653, 2015.

\bibitem{Tale}
G.~Talenti.
\newblock Boundedness of minimizers.
\newblock {\em Hokkaido Math. J.}, {\bf 19}(2):259--279, 1990.

\end{thebibliography}

\end{document}